\documentclass[11pt]{article}
\usepackage[english]{babel}
\usepackage{amsmath,amssymb}
\usepackage[colorlinks=true]{hyperref}
\hypersetup{urlcolor=blue, citecolor=red, linkcolor=blue}
\usepackage{amsthm}

\usepackage{yhmath} 

\usepackage[T1]{fontenc}

\usepackage[a4paper, left=25mm, right=25mm, top=30mm, bottom=30mm]{geometry}

\usepackage{mathtools}
\usepackage{mathrsfs}
 
\theoremstyle{plain}
\newtheorem{theorem}{Theorem}[section]

\newtheorem{proposition}[theorem]{Proposition}
\newtheorem{lemma}[theorem]{Lemma}

\theoremstyle{definition}
\newtheorem{definition}[theorem]{Definition}
\newtheorem{remark}[theorem]{Remark}

\newcommand\xqed[1]{%
    \leavevmode\unskip\penalty9999 \hbox{}\nobreak\hfill
	\quad\hbox{#1}}
\newcommand\demo{\xqed{$\triangle$}}

\newtheorem{example}[theorem]{Example}

\newcommand\restr[2]{{
  \left.\kern-\nulldelimiterspace 
  #1 
  \right|_{#2} 
}}

\newcommand{\R}{\mathbb{R}}

\renewcommand{\d}{\mathrm{d}}

\newcommand{\Cinfty}{\mathscr{C}^\infty}
\newcommand{\T}{\mathrm{T}}
\newcommand{\cT}{\mathrm{T}^\ast}

\newcommand{\X}{\mathfrak{X}}

\newcommand{\parder}[2]{\frac{\partial #1}{\partial #2}}

\DeclareMathOperator{\ad}{ad}

\DeclareMathAlphabet{\mathpzc}{OT1}{pzc}{m}{it}
\def\d{\mathrm{d}}

\numberwithin{equation}{section}

\title{{\sffamily Hamiltonian stochastic Lie systems and applications}}

\author{{\sffamily $^a$ Eduardo Fern\'andez-Saiz
\thanks{e-mail:
   e.fernandezsaiz@cunef.edu  \ ORCID: 0000-0002-6220-4987}\ ,\
$^b$Javier de Lucas%
\thanks{e-mail:
   javier.de.lucas@fuw.edu.pl \ ORCID: 0000-0001-8643-144X (Corresponding author)}\ ,\
$^c$Xavier Rivas%
\thanks{e-mail:
   xavier.rivas@urv.cat \ ORCID: 0000-0002-4175-5157}, and\
$^b$Marcin Zaj\k ac%
\thanks{e-mail:
   marcin.zajac@fuw.edu.pl \ ORCID: 0000-003-4914-3531}
}
\\[1ex]
\normalsize\itshape\sffamily
$^a$Department of Mathematics, CUNEF Universidad, \\
\normalsize\itshape\sffamily
C. de Leonardo Prieto Castro 2, 28040 Madrid,
Spain
\\[1ex]
\normalsize\itshape\sffamily
$^b$
\normalsize\itshape\sffamily
Department of Mathematical Methods in Physics, University of Warsaw, \\\normalsize\itshape\sffamily   ul. Pasteura 5, 02.093, Warszawa, Poland
\\[1ex]
\normalsize\itshape\sffamily
$^c$Department of Computer Engineering and Mathematics, Universitat Rovira i Virgili,\\
\normalsize\itshape\sffamily
Avinguda Països Catalans 26, 43007, Tarragona, Spain
\\[1ex]
}

\begin{document}

\maketitle
\begin{abstract}

This paper provides a practical approach to stochastic Lie systems, i.e.  stochastic differential equations whose general solutions can be written as a function depending only on a generic family of particular solutions and some constants related to initial conditions. We correct the stochastic Lie theorem characterising stochastic Lie systems, proving that, contrary to previous claims, it retains its  classical form in the Stratonovich approach. Meanwhile, we show that the form of stochastic Lie systems may significantly differ from the classical one in the It\^{o} formalism. New generalisations of stochastic Lie systems, like the so-called stochastic foliated Lie systems, are introduced. Subsequently, we focus on stochastic Lie systems that are Hamiltonian systems relative to different geometric structures. Special attention is paid to the symplectic case. We study their stability properties and lay the foundations of a stochastic energy-momentum method. A stochastic Poisson coalgebra method is developed to derive superposition rules for Hamiltonian stochastic Lie systems. Potential applications of our results are presented for biological stochastic models, stochastic oscillators, stochastic Lotka--Volterra systems, Palomba--Goodwin models, among others. Our findings complement previous approaches by using stochastic differential equations instead of deterministic equations designed to capture some of the features of models of stochastic nature.
\end{abstract}

\bigskip

(Primary),
37N25, 
53Z05 
(Secondary).

\newpage

{\setcounter{tocdepth}{2}
\def\baselinestretch{1}
\small
 
\def\addvspace#1{\vskip 1pt}
\parskip 0pt plus 0.1mm
\tableofcontents
}

\section{Introduction}

In its most classical definition, a {\it superposition rule} is a $t$-independent function that describes the general solution of a $t$-dependent system of first-order ordinary differential equations (ODEs) in normal form, a so-called {\it Lie system}, via a generic family of its particular solutions and a set of constants related to initial conditions \cite{BM10,CGM00,CGM07,PW83}. Superposition rules are used, for instance, in approximate and numerical methods, as they are applicable to Lie systems whose exact general solutions in explicit form are unknown \cite{Dissertationes,Pi12,PW83}. In particular, the knowledge of a particular finite set of exact and/or approximate solutions of a Lie system permits one to study its general solution via superposition rules \cite{PW83}.  

The superposition rule concept traces its origins back to Sophus Lie's pioneering and celebrated book \cite{Lie1893},  edited by Georg Scheffers. In that work, Lie stated the theorem nowadays called the Lie theorem, characterising Lie systems. Prior to that, Lie briefly described his Lie theorem in \cite{Li93a} without a proof, as a criticism of some previous works on superposition rules by Vessiot, Guldberg, K\"oningsberger, and other researchers (see \cite{LS21,Gu93,Ve93} and references therein). Lie stated that the results on superposition rules for differential equations on $\mathbb{R}$ devised by previous authors were a simple application of his theory on infinitesimal  transformation groups. Moreover, \cite{Lie1893} laid down the foundations for the theory of Lie systems. The Lie theorem is also called Lie--Scheffers theorem, as Scheffers took part in editing  Lie's work \cite{Lie1893}, or Lie superposition theorem \cite{BM10local-global}.  Previous remarks suggest saying `Lie theorem' rather than `Lie--Scheffers theorem' and using the denomination `Lie system' instead of `Lie--Scheffers system'. Since Vessiot made important contributions to the theory of Lie systems \cite{Dissertationes,LS21}, the term `Lie--Vessiot system' is also an appropriate designation for a Lie system. Scheffers, however, never independently researched Lie systems, and it is quite unlikely that he established any findings on the subject on his own. \cite{LS21}. 

The Lie theorem states that a $t$-dependent system of ODEs in normal form admits a superposition rule if and only if it describes the integral curves of a $t$-dependent family of vector fields that can be viewed as a curve in a finite-dimensional Lie algebra of vector fields, a {\it Vessiot--Guldberg Lie algebra} (see \cite{BM10local-global,CGM00,CGM07} for modern approaches  and further details). 

Lie systems have been thoroughly studied due to their widespread occurrence in physics and mathematics (see \cite{Dissertationes,LS21}, which contain more than 200 references on Lie systems and related topics). In the 1980s, Winternitz and his colleagues at the Centre de Recherches Math\'ematiques of the University of Montreal conducted an extensive study of Lie systems. In subsequent years, while Winternitz shifted his focus to other subjects, some of his collaborators continued exploring the topic \cite{OG_00}. Additionally, scholars from Poland, Italy, Spain,  Mexico, and Russia, such as J. Grabowski, G. Marmo, J. F. Cari\~nena, J. de Lucas,  F. J. Herranz, R. Flores-Spinoza, Y. Vorobiev and N. H. Ibragimov together with their research teams, began contributing to this field (see \cite[Chapter 1]{LS21}, \cite{Dissertationes} and references therein).

There has been a vast effort to generalise Lie systems to more general situations: $t$-dependent Schr\"odinger equations \cite{CR_04,Dissertationes}, partial differential equations \cite{OG_00}, quasi--Lie systems \cite{Dissertationes}, foliated Lie systems \cite{CGM00,Wys_23}, discrete differential equations \cite{BJLS_24,PW_04}, stochastic differential equations \cite{LO09}, superdifferential equations \cite{BGHW_87}, and others \cite{Dissertationes}. There has also been much interest in describing the geometrical properties of Lie systems and in using them to study differential geometric problems (see \cite{AHKR_20,AHR_22,Car_25,LS21,FS_25} and references therein). Moreover, Lie systems are related to important physical and mathematical models, which strongly motivates their analysis \cite{CGM00,GIMM_15,Dissertationes,Flo_11,Gam_17}.
In this work, we are mainly concerned with the extension of superposition rules and Lie systems to the realm of stochastic differential equations \cite{LO09}. In this way, we aim to draw the attention of researchers working on Lie systems to stochastic models, and vice versa. Therefore, to enhance accessibility to our work, we will provide a concise overview of various geometric and stochastic concepts.

Stochastic differential equations may describe phenomena that deterministic differential equations cannot \cite{Al09}. For example, the probability of disease extinction or outbreak, the quasi-stationary probability distribution, the final size distribution, and the expected duration of an epidemic are features that cannot effectively be modelled by deterministic methods \cite{Al09}. These and other reasons motivate the great interest in studying stochastic differential equations. Nevertheless, some deterministic differential equations can capture certain interesting characteristics of models even more easily than stochastic differential equations \cite{EFSZ20}, and stochastic differential equations offer a complementary alternative view \cite{GG_09}.

It is interesting to extend superposition rules to stochastic differential equations. The work \cite{LO09} extends the notion of a superposition rule to a class of systems of stochastic first-order ordinary differential equations, called {\it stochastic Lie systems}. The authors use a Stratonovich approach, as this makes the theory similar to the deterministic theory of Lie systems \cite{CGM00,Dissertationes,LS21}. Moreover, \cite{LO09} presents a precise and interesting account of certain local results about stochastic Lie systems and, as a byproduct, it also explains many technical results on standard Lie systems, which are usually absent in the literature \cite{CGM00,CGM07,LS21,Ram_02}. Despite its mathematical interest, many of the technical details given in \cite{LO09} are frequently omitted, as they generally have few practical applications. As noted in \cite{BM10local-global,LS21}, standard works on Lie systems, even theoretical ones, are essentially interested in local aspects and generic points, which leads them to skip many technical proofs analysed in \cite{LO09}. Nevertheless, there is mathematical interest in both global and local technical aspects, as illustrated by \cite{BM10local-global} for global superposition rules and by \cite{LO09}. It is worth stressing that stochastic Lie systems were found to have applications in the description of Brownian motions, economic models such as the Black--Scholes theory of derivative pricing, and so on \cite{LO09}. The potential interest of stochastic Lie systems in epidemic models was very briefly mentioned in \cite{EFSZ20}, without, as far as we know, any further development. It is worth noting that there are many new potential applications of the theory of Lie systems to stochastic models, which have so far remained almost unexplored.

In \cite{LO09}, a stochastic Lie theorem characterising Stratonovich stochastic first-order ordinary differential equations admitting a superposition rule was devised, but it contains a mistake that changes its meaning and potential applications. More precisely, the direct part of the stochastic Lie theorem in \cite{LO09} states that a stochastic Lie system may admit, locally, a superposition rule if its associated Stratonovich operator is related to a family of vector fields spanning an involutive distribution. Our present work proves that the vector fields must additionally span a finite-dimensional real Lie algebra of vector fields, which is a much stronger condition that already appears in the classical Lie--Scheffers theorem \cite{CGM00,CGM07,Lie1893}. We also determine the precise point of the mistake in \cite{LO09}, correct the statement of the stochastic Lie theorem, and provide a counterexample, based on SIS epidemic models, to illustrate when the stochastic Lie theorem in \cite{LO09} fails.
It is worth stressing that the differences between the incorrect and the correct versions of the stochastic Lie theorem have no impact on the applications carefully studied in \cite{LO09}.

Moreover, our work presents a concise introduction to stochastic Lie systems, aiming to provide a practical approach while avoiding technical details that are not necessary for general purposes. In this sense, it simplifies the elegant and mostly rigorous mathematical treatment in \cite{LO09} by using standard assumptions in mathematical constructions. For instance, we focus on local results at generic points, which significantly simplifies previously required techniques.

Our correction of the stochastic Lie theorem shows that the stochastic Lie theorem has no exclusive stochastic features in the Stratonovich approach: it retains the conditions of the classical Lie theorem. Meanwhile, it should be stressed that the conditions for a stochastic differential equation in It\^{o} form \cite{Ka81} to become a stochastic Lie system do not follow the standard form expected from the deterministic Lie theorem \cite{CGM00,CGM07,Dissertationes}. This is very important in practice, as many relevant stochastic differential equations are given in It\^{o} form and must be translated into the Stratonovich approach \cite{Sa70} in order to apply the methods of our work and of \cite{LO09}. In this respect, the relation between the It\^{o} and the Stratonovich approaches is reviewed, and some examples with potential applications are provided in this work. It is worth noting that stochastic differential equations in It\^{o} form may appear to be stochastic Lie systems, but they are not. This occurs in certain SIS epidemiological models \cite{GGHMP11}, as shown in this work.

Apart from introducing stochastic Lie systems, this paper suggests the applicability of various generalisations of Lie systems \cite{Dissertationes} to the stochastic domain. This opens a new vast realm of potential applications in physical, mathematical, and biological models, offering promising avenues for further exploration. One example is the extension of the theory of foliated Lie systems \cite{Wys_23} to the realm of stochastic differential equations. In fact, we suggest that, in analogy with the deterministic case, this may arise when studying certain problems of our stochastic energy-momentum method devised here \cite{MS88}.

By the stochastic Lie theorem, every stochastic Lie system admitting $\ell$ independent random variables is related to a family of $\ell$ vector fields (if time is considered as a special `deterministic' random variable) parametrised by random variables. All of them can be understood as linear combinations of elements of a finite-dimensional Lie algebra of vector fields (a so-called Vessiot--Guldberg Lie algebra) with coefficients depending on the random variables. More particularly, we study stochastic Lie systems admitting a Vessiot--Guldberg Lie algebra of Hamiltonian vector fields relative to some compatible differential geometric structure: the {\it Hamiltonian stochastic Lie systems}. In particular, we mainly focus on Hamiltonian stochastic Lie systems relative to symplectic forms, although our theory is easily generalisable to other geometric structures and stochastic Lie systems. In this context, the coalgebra method is extended to Hamiltonian stochastic Lie systems to derive superposition rules. This provides an extension to the stochastic realm of the theory of Hamiltonian Lie systems and their generalisations \cite{LS21}.

Our results are illustrated with many new examples of stochastic Lie systems. In particular, we study SIS models \cite{GGHMP11}. SIS models are epidemiological models that assume that individuals do not acquire immunity after infection. They concern two variables/compartments: $S$, representing susceptible individuals, and $I$, representing infected individuals in a large population of size $N$ where a single disease is spreading. SIS systems are usually treated in the literature in a deterministic manner \cite{EFSZ20}. This can be used to describe some of their features, but not all, as some are purely stochastic in nature.
Our models can also be used to study stochastic models arising in Lotka--Volterra systems, Palomba--Goodwin models \cite{Pal_39}, stochastic oscillators \cite{HH11,Ha19}, etc. In fact, stochastic Lie systems in general, and Hamiltonian stochastic Lie systems in particular, seem to have a wide range of potential applications.

We are also concerned with a theory of stability for Hamiltonian stochastic systems and, in particular, Hamiltonian stochastic Lie systems. Our study is specially concerned with linear ones, which appear as approximations of nonlinear ones and share some stability properties with them \cite{Ar74}. An example of a stochastic oscillator with a drift is analysed. Moreover, the basis for an energy-momentum method \cite{MS88} for Hamiltonian stochastic differential equations is established by using some results in \cite{LO08}. In particular, this allows one to study the relative equilibrium points of Hamiltonian systems, i.e. points where the dynamics is generated by Hamiltonian symmetries of the system under study. A characterisation (see Theorem \ref{RelativeEquilibrum}) of the relative equilibrium points for stochastic Hamiltonian systems in terms of critical points of their Hamiltonian functions is developed. As a byproduct, strong constants of motion \cite{LO08} for Hamiltonian stochastic Lie systems are briefly studied and illustrated with examples.

The structure of the paper is as follows.
Section~\ref{Fundamentals} is a brief introduction to stochastic differential equations, stochastic Lie systems, and many other notions to be used in this paper. It also shows the difference between the form of stochastic Lie systems in the Stratonovich and the It\^{o} approaches. Section~\ref{Sec:SSRandLT} is concerned with superposition rules for stochastic Lie systems and reviews and corrects the previous version of the stochastic Lie theorem. Section \ref{Sec:Hamiltonian_Stochastic_Lie_Systems} deals with Hamiltonian stochastic Lie systems. In particular, we define the newly proposed stochastic foliated Lie systems and the Hamiltonian counterparts of stochastic Lie systems. A theory of stability for stochastic Lie systems is given in Section \ref{Sec:SHS}. A relative equilibrium notion is presented and studied in Section \ref{Sec:Relative_Equilibrium_Stochastic_Differential}, and a stochastic version of a classical result characterising relative equilibrium points is presented. Meanwhile, Section \ref{Se:PCSLS} develops the Poisson coalgebra method for Hamiltonian stochastic Lie systems. Finally, Section \ref{Eq:Conclusions} presents our conclusions and future work.

\section{Stochastic differential
equations and stochastic Lie systems}\label{Fundamentals}

This section provides a concise introduction to stochastic differential equations and stochastic Lie systems. We have tried to provide enough information to follow the paper for people working on stochastic differential equations or  the theory of Lie systems. 

A detailed survey on the theory of stochastic differential equations can be found in \cite{Ar74, E82}, while the theory of stochastic Lie systems was elaborated for the first time in \cite{LO09}, which offers a precise formulation of the theory. To avoid technical details, we will assume objects to be smooth, locally defined, and problems at generic points satisfying very mild conditions. Following the classical Lie systems theory in \cite{Dissertationes,LS21}, we here provide a definition of a stochastic Lie system not based on the notion of a superposition rule.

In a nutshell, a $t$-dependent system of first-order ordinary differential equations on an $n$-dimensional manifold $M$ of the form
\begin{equation}\label{eq:Det}
\frac{\d \Gamma^i}{\d t}=X^i(t,\Gamma)\,,\qquad i=1,\ldots,n\,,\qquad \forall t\in \mathbb{R}\,,\qquad \forall \Gamma\in M\,,
\end{equation}
for certain functions $X^1,\dotsc,X^n\in \Cinfty(\mathbb{R}\times M)$,
 is deterministic in the sense that an initial condition in $M$  at a time $t_0$ establishes, under mild conditions on the functions $X^1,\ldots,X^n$, a unique solution giving a position in $M$ for every $t\in \mathbb{R}$. Geometrically, the coefficients $X^1(t,\Gamma),\ldots,X^n(t,\Gamma)$ give rise to a $t$-parametrised vector field on $M$ of the form
 $$
X=\sum_{i=1}^n X^i(t,\Gamma)\frac{\partial}{\partial \Gamma^i},
 $$
 which is formally called a {\it  $t$-dependent vector field}. More generally, an $\mathbb{R}^\ell$-dependent vector field on a manifold $M$ is a family of vector fields on $M$ parametrised by elements of  $\mathbb{R}^\ell$ (see \cite{Dissertationes} for details). 
 
 Then, \eqref{eq:Det} may be modified by considering that its form can also depend on certain `stochastic processes' $B^1,\ldots,B^r$, namely a series of $t$-dependent random variables satisfying certain appropriate conditions to be set hereafter in detail. This fact is shown by the expression
\begin{equation}\label{Eq:practicalSDE}
\delta \Gamma^i=X^i_1(B,\Gamma)\delta t+\sum_{\alpha=2}^\ell X^i_\alpha(B,\Gamma)\circ \delta B^\alpha\,,\qquad i=1,\ldots,n\,,
\end{equation}
for functions $ X^i_1 ,\ldots,X_\ell^i\in \Cinfty( \mathbb{R}^\ell\times M)$, with $B=(t,B^2,\ldots,B^\ell)$, $\Gamma=(\Gamma^1,\ldots,\Gamma^n)$, and $i=1,\ldots,n$. The symbol $\circ$ has been used to indicate that \eqref{Eq:practicalSDE} is understood in the so-called Stratonovich interpretation, to be briefly explained afterwards.
More precisely, $(\Omega, \mathcal{F}, P)$ is a {\it probability space}, where $\Omega$ is a manifold, $\mathcal{F}$ is a $\sigma$-algebra of subsets of $\Omega$, and $P:\mathcal{F}\rightarrow [0,1]$ is a probability function on $\mathcal{F}$. Each $B_\alpha:\mathbb{R}_+\times\Omega\rightarrow \mathbb{R}$ is a {\it semi-martingale} for $\alpha=1,\ldots,r$. Semi-martingales are good integrators relative to the It\^{o} or the Stratonovich integrals due to their properties. In short, a {\it martingale}  is a sequence of stochastic processes such that, at a particular time, the conditional expectation of the next value is equal to the present value, independently of all previous values.

A stochastic differential equation is then an expression on a manifold $M$ of the form
\begin{equation}
  \label{stochastic differential equation expression} \delta \Gamma = \mathfrak{S} (B,
  \Gamma)\circ \delta B\,,
\end{equation}
where $B:\mathbb{R}_+ \times \Omega \rightarrow \R^\ell$ is an
$\R^\ell$-valued semi-martingale and $\mathfrak{S} (B, \Gamma) : \T_B \mathbb{R}^\ell
\to \T_\Gamma M$, with $(B,\Gamma)\in \mathbb{R}^\ell\times M$, describes a Stratonovich operator. Every basis in
$\T^{\ast} \R^\ell$ allows one to decompose $\mathfrak{S}(B,
\Gamma)$ into $\ell$-components $(\mathfrak{S}_1 (B, \Gamma), \ldots, \mathfrak{S}_\ell (B, \Gamma))$ in the chosen basis, which will be frequently employed hereafter. Geometrically, every component is a mapping 
$$
\mathfrak{S}_\alpha: (B, \Gamma)\in \mathbb{R}^\ell\times M\mapsto\mathfrak{S}_\alpha(B,\Gamma)\in \T_\Gamma M\subset \T M,\qquad \alpha=1,\ldots,\ell,
$$
which gives rise, for every fixed value of $B$, to a vector field on $M$. In fact, $\mathfrak{S}_{\alpha}(B,\cdot):\Gamma\in M\mapsto \mathfrak{S}_\alpha(B,\Gamma)\in \T_\Gamma M\subset \T M$. In other words, $\mathfrak{S}_\alpha$ can be understood as an $\mathbb{R}^\ell$-parametrised vector field, which is usually called an $\mathbb{R}^\ell$-dependent vector field \cite{Dissertationes}.
Every particular solution to \eqref{stochastic differential equation expression} is also a semi-martingale $
\Gamma:\mathbb{R}_+\times \Omega\rightarrow M$. Moreover, we say that a particular solution has initial condition $\Gamma_0
\in M$ when $\Gamma(0,\omega_0)=\Gamma_0$ for every $\omega_0\in \Omega$ with probability one. Note that the standard time can be considered as a random variable $t:(t,\omega_0)\in \mathbb{R}_+\times \Omega\mapsto t\in\mathbb{R}$, whose value is independent of $\Omega$ and is included as the first component of $B$. In practice, $B$ is related to Brownian motions, also called Wiener processes, c\`adl\'ag martingales, It\^o processes of the form $\delta X = \sigma 
\delta W + \mu
\delta t$ for a Brownian motion $W$ and adapted processes $\sigma, \mu$, L\'evy processes, etc. It is worth noting that semi-martingales form the largest class of processes for which the It\^o integral can be defined. On the other hand, white noises are not semi-martingales. 

We assume that the driving processes are adapted to the natural filtration $(\mathcal{F}_t)_{t\geq 0}$, i.e. we only allow processes that evolve consistently with the information available up to the present time. Vector fields will satisfy local Lipschitz and linear growth conditions (or equivalent geometric hypotheses) guaranteeing existence and uniqueness of solutions. These conditions are a stochastic analogue of the usual ODE assumptions, ensuring existence and uniqueness of solutions to the SDE. Unless otherwise stated, the semi-martingales considered have continuous trajectories (in particular, Brownian motions). If jump processes such as L\'evy noise are allowed, the canonical Stratonovich calculus is not well defined; in that case, the Marcus integral provides the appropriate extension. Since our results rely on the Stratonovich calculus on smooth manifolds, we restrict throughout to continuous semi-martingale drivers.

The solution to \eqref{stochastic differential equation expression} is  of the form
\begin{equation}\label{Eq:SolStr}
\Gamma(t)-\Gamma(0)=\int_0^t\mathfrak{S}_1(B,\Gamma)\delta t+\sum_{\beta=2}^{\ell}
\int_0^t\mathfrak{S}_\beta(B,\Gamma) \circ \delta B^\beta_t,
\end{equation}
where the integrals appearing above are {\it Stratonovich integrals}. It is relevant to understand that stochastic differential equations can still be understood in the so-called It\^{o} way, in which the general solution is also of the form \eqref{Eq:SolStr}, but the integrals are assumed to be It\^{o} integrals, which is different. Which approach is used depends on the applications to be developed. It should be observed that in the interpretation of stochastic differential equations according to It\^o, the symbol $\circ$ is omitted. 
 
Unless otherwise stated, stochastic differential equations are assumed to be in the Stratonovich sense. This is motivated by the {\it Malliavin transfer principle} \cite{Lea_06,Ma78}, which suggests that the obtained theory will retain the standard differential theory, although it may not always be the case. In other words, Malliavin transfer principle states that ``A formula which is true in the deterministic context and which has a meaning via
Stratonovich stochastic calculus, is still valid, but only almost surely'' \cite{Lea_06}. 

Despite the Malliavin transfer principle, stochastic differential equations appear in the It\^{o} sense in many applications \cite{EFSZ22,NM19,WCDG18}. Hence, a manner to deal with such stochastic differential equations will be studied in this work. Moreover, the analysis of such stochastic differential equations shows some differences appearing in the stochastic formalism with respect to the deterministic theory of Lie systems, which enriches the theory.

Although every Stratonovich stochastic differential equation is equivalent to another It\^{o} stochastic differential equation, the form of both is different. Hence, one has to take care of the method employed to study a stochastic differential equation \cite{Ev13}. More exactly, assume that \eqref{Eq:practicalSDE} has coefficients that do not depend on $ {B}^2,\ldots, {B}^\ell$, then the It\^{o} differential equation admitting the same solutions  reads (see \cite[p. 137]{BCC_01} for details)
\begin{equation}\label{Eq:TranItoStratonovich}
\delta \Gamma^i=\left(\mathfrak{S}^i_1(t,\Gamma)+\frac 12\sum_{\beta=2}^\ell\sum_{j=1}^n\frac{\partial \mathfrak{S}_\beta^i}{\partial \Gamma^j}(t,\Gamma)\mathfrak{S}_\beta^j(t,\Gamma)\right)\delta t +\sum_{\beta=2}^\ell \mathfrak{S}^i_\beta(t,\Gamma )\delta {B}^\beta_t,\qquad i=1,\ldots,n.
\end{equation}
Note that we have dropped the $\circ$ sign in the previous expression.
It is worth stressing that the definition of a stochastic Lie system to be given soon relies on the new term appearing multiplying $\delta t$ in \eqref{Eq:TranItoStratonovich}, which is called the {\it drift term} \cite{BCC_01}. Moreover, the transformation from the Stratonovich to the Itô form does not change the coefficients with the $\delta {B}^\beta_t$ for $\beta = 2,\dotsc,\ell$.

\begin{example}\label{Ex:FirstItoStr} Let us consider a stochastic differential equation induced by a semi-martingale $B:(t,z)\in \mathbb{R}_+\times \Omega \mapsto (t,\mathcal{B})\in \mathbb{R}^2$ consisting of two variables (a  deterministic one, $t$, which can be understood as a particular type of stochastic variable, and a purely stochastic one, $\mathcal{B}$, describing a Brownian motion) of the It\^{o} form
\begin{equation}\label{eq:It}
\delta I=I(\beta N-\mu-\gamma-\beta I)\delta t+I\sigma (N-I) \delta \mathcal{B}\,,
\end{equation}
where $N$ is a constant describing the total population of a SIS epidemiological system \cite{GGHMP11}. Let us assume that $\sigma=\sigma(t)$. This model  describes the dissemination of a single communicable disease in a susceptible population of size $N$ (see \cite{GGHMP11}).  

The Stratonovich stochastic differential equation related to \eqref{eq:It} takes the form
\begin{equation}\label{eq:St}
\delta I=\left(-\sigma^2(t) I^3+\left(-\beta+\frac{3N\sigma^2(t)}{2}\right)I^2+\left(N \beta-\gamma-\mu-\frac{\sigma^2(t)N^2}{2}\right)I\right)\delta t+\sigma(t)  (N-I)I\circ \delta \mathcal{B}\,.
\end{equation}
This induces a Stratonovich operator between $\mathbb{R}^2$ and $\mathbb{R}$ such that, for each $t,\mathcal{B},I$, one has an operator
$$ \mathfrak{S}(t,\mathcal{B},I):(\delta t,\delta \mathcal{B})\in \T_{(t,\mathcal{B})}\mathbb{R}^2\longmapsto \mathfrak{S}_1(t,I)\delta t+\mathfrak{S}_2(t,I)\circ \delta \mathcal{B}\in \T_I\mathbb{R}\, $$
for 
$$
\begin{gathered}
\mathfrak{S}_1(t,I)=-\sigma^2(t) I^3+\left(-\beta+\frac{3N\sigma^2(t)}{2}\right)I^2+\left(N \beta-\gamma-\mu-\frac{\sigma^2(t)N^2}{2}\right)I,\\
\mathfrak{S}_2(t,I)=\sigma(t) (N-I)I.
\end{gathered}
$$
Geometrically, each component of $\mathfrak{S}$ can be understood as an $\mathbb{R}^2$-dependent vector field. More specifically, $\mathfrak{S}_1$ and $\mathfrak{S}_2$ can be understood as $\mathbb{R}^2$-dependent vector fields
$$
\left[-\sigma^2(t) I^3+\left(-\beta+\frac{3N\sigma^2(t)}{2}\right)I^2+\left(N \beta-\gamma-\mu-\frac{\sigma^2(t)N^2}{2}\right)I\right]\frac{\partial}{\partial \Gamma},\qquad \sigma(t) (N-I)I\frac{\partial}{\partial \Gamma}.
$$
Solutions to \eqref{eq:St} are given by semi-martingales of the form $\Gamma:(t,\omega_0)\in \mathbb{R}_+\times \Omega\mapsto I(t,\omega_0)\in \mathbb{R}$ determined by an initial condition described by a random variable $\Gamma_0:\Omega\rightarrow \mathbb{R}$ such that $
\Gamma_0(\omega_0)=1$ for every $\omega_0\in \Omega$. \demo
\end{example}

Equations \eqref{Eq:TranItoStratonovich} and the last example illustrate the central role of the
Stratonovich formulation. To make the relation with the Itô formulation explicit,
we state the following standard lemma \cite{Kunita1990, Oksendal2003}.
\begin{lemma}[Itô–Stratonovich correction in coordinates]
\label{lem:ItoStratCoords}
Let $X=(X^1,\dots,X^n)$ be a Stratonovich diffusion on a smooth manifold $M$ with local
coordinates $(x^i)$, written as
\[
\mathrm{d}x^i_t = A^i(x_t)\,\mathrm{d}t + \sum_{k=1}^m B^{i}_{k}(x_t)\circ \mathrm{d}B^k_t .
\]
Then the corresponding Itô form is
\[
\mathrm{d}x^i_t = \Bigl(A^i(x_t) + \tfrac12 \sum_{k=1}^m \sum_{j=1}^n 
  B^j_{k}(x_t)\,\frac{\partial B^i_{k}}{\partial x^j}(x_t)\Bigr)\mathrm{d}t
  + \sum_{k=1}^m B^{i}_{k}(x_t)\,\mathrm{d}B^k_t .
\]
Equivalently, in tensorial notation
\[
X^{\text{Itô}}_0 \;=\; X^{\text{Strat}}_0 + \tfrac12 \sum_{k=1}^m \nabla_{X_k}X_k,
\]
independently of the choice of torsion-free connection $\nabla$.
\end{lemma}

The additional drift term $\tfrac12\sum_{k=1}^m \nabla_{X_k}X_k$ may not belong to the
finite-dimensional Lie algebra generated by $\{X_0,\dots,X_m\}$.
Therefore, a system that is a Lie system in the Stratonovich sense may cease to be one in the Itô
formulation. This observation is important for applications, where many models are naturally written
in Itô form, and motivates our systematic use of the Stratonovich framework in the theory of
stochastic Lie systems.

Recall that the theoretical utilisation of Stratonovich stochastic equations is justified by Malliavin's Transfer Principle \cite{Ma78}, which states that the results from the theory of ordinary differential equations remain applicable, in an analogous way, to stochastic differential equations in Stratonovich form. This principle is just a general guideline without a proof, which implies that it must be used just as a general suggestion. Nevertheless, as shown in Example \ref{Ex:FirstItoStr}, the relation to the It\^{o} approach has to be  too   considered for analysing applications.

Let us now turn to analysing the main type of stochastic differential equations to be studied in this work (see \cite{LO09} for the pioneering work on the topic).

\begin{definition} A {\it stochastic Lie system} is a stochastic differential equation on a manifold $M$ of the form $
\delta \Gamma=\mathfrak{S}(B,\Gamma)\circ \delta B
$ such that $B:\mathbb{R}_+\times \Omega\rightarrow \mathbb{R}^\ell$ is a semi-martingale and $\mathfrak{S}$ is a Stratonovich operator such that 
\begin{equation}
\label{Eq:STL}\mathfrak{S}(B,\Gamma)=\left(\sum_{\alpha=1}^rb_1^\alpha(B)X_\alpha(\Gamma),\ldots,\sum_{\alpha=1}^rb_\ell^\alpha(B)X_\alpha(\Gamma)\right),\qquad \forall\Gamma\in M,\quad \forall B\in \mathbb{R}^\ell,
\end{equation}
for a family of functions $b_a^\alpha:B\in \mathbb{R}^\ell\mapsto b_a^\alpha(B)\in \mathbb{R}$ that are assumed to be non-anticipative, i.e. measurable with
respect to the natural filtration, with $\alpha=1,\ldots,r$ and $a=1,\ldots,\ell$, and a certain $r$-dimensional Lie algebra of vector fields on $M$ spanned by $X_1,\ldots,X_r$.    We call the Lie algebra $V=\langle X_1,\ldots,X_r\rangle$ a {\it Vessiot--Guldberg Lie algebra} of the stochastic Lie system \eqref{Eq:STL}.
\end{definition}

As for any other Stratonovich stochastic differential equation \eqref{stochastic differential equation expression}, recall that the Stratonovich operator $\mathfrak{S}=(\mathfrak{S}_1,\ldots,\mathfrak{S}_\ell)$ of a stochastic Lie system can be considered as an $\ell$-element family consisting of $\mathbb{R}^\ell$-vector fields on $M$ of the form $\mathfrak{S}_\alpha:(B,\Gamma)\in \mathbb{R}^\ell\times M\mapsto \mathfrak{S}_\alpha(B,\Gamma)\in \T_\Gamma M\subset \T M$ for $\alpha=1,\ldots,\ell$. Moreover, each one of these $\mathbb{R}^\ell$-vector fields on $M$ is a linear combination with functions depending on $\mathbb{R}^\ell$ of a finite-dimensional Lie algebra of vector fields $X_1,\ldots,X_r$. Recall that this implies that there exists constants $c_{\alpha\beta}^\gamma$, with $\alpha,\beta,\gamma=1,\ldots,r$, so that
$$
[X_\alpha,X_\beta]=\sum_{\gamma=1}^rc_{\alpha\beta}^\gamma X_\gamma,\qquad \alpha,\beta=1,\ldots,r.
$$
It is very important to stress that $c_{\alpha,\beta}^\gamma$ are constants. Although a stringent condition, it is justified by the significant applications and geometric characteristics of Lie systems (refer to \cite{Dissertationes,LS21} where over a two hundred related works, applications, and key authors of Lie systems are cited) and stochastic counterparts in \cite{LO_09} and this work. Let us say in advance that, as shown in following parts of this work, the dimension and nature of the Vessiot--Guldberg Lie algebra is related to the properties of the stochastic Lie system. In standard Lie systems, where $B$ is just the time, a solvable Vessiot--Guldberg Lie algebra ensures, for instance, that the associated Lie system can be integrated by quadratures \cite{CRFG_16}.


\begin{example}\label{eq:WhiteOscil} Let us consider a generalisation of a damped harmonic oscillator on $\T \mathbb{R}$ with a stochastic part  modelled by means of a semi-martingale $W_1$ related to a one-dimensional Wiener process in It\^o form, which retrieves as a particular cases several models in the previous literature (cf. \cite{HH11,Ha19,MH_04}). In particular, consider adapted coordinates $(x,y=\dot x)$ on $\T \mathbb{R}$ and the stochastic differential equation given by
\begin{equation}\label{eq:HOWN}
\delta\begin{pmatrix} x \\ y \end{pmatrix} = \begin{pmatrix} 0 & 1 \\ -\omega_f^2(t) & -k(t) \end{pmatrix}\begin{pmatrix} x \\ y \end{pmatrix}\delta t + \begin{pmatrix} 0 & 0 \\ 0 & -\sigma(t) \end{pmatrix}\begin{pmatrix} x \\ y \end{pmatrix} \delta W_1\,,
\end{equation}
where $\sigma(t)$ is any $t$-dependent function, e.g. quantifying noise, and the functions $\omega_f(t)$ and $k(t)$ are extensions to a stochastic realm of the usual functions relative to the standard deterministic model for a dissipative harmonic oscillator
$$ \ddot x+\omega_f^2(t) x+k(t)\dot x=0\,. $$
Then, $\omega_f(t)$ is a $t$-dependent frequency  of the oscillator \eqref{eq:HOWN} while $k(t)$ is frequently used to describe a friction-like effect.

Let us see how the model \eqref{eq:HOWN} can be considered as a linear stochastic Lie system. The first step needed to apply our formalism is  to transform  the previous stochastic system from an It\^o into a Stratonovich one. This can be reached in a simple manner by applying the transformation equation given in \eqref{Eq:TranItoStratonovich}.  Indeed, this expression shows that \eqref{eq:HOWN}  can be related to a Stratonovich operator of the form
\begin{equation}\label{Eq:Str}
\mathfrak{S}(t, W_1, x, y)=\left(\begin{pmatrix} 0 & 1 \\ -\omega_f^2(t)& -k(t) -\sigma^2(t)/2 \end{pmatrix}\begin{pmatrix}x\\y\end{pmatrix},\begin{pmatrix}0 & 0 \\ 0 & -\sigma(t)\end{pmatrix}\begin{pmatrix}x\\y\end{pmatrix}\right). 
\end{equation}
It is worth recalling that many stochastic differential equations are formulated in the It\^o framework. 

Once the Stratonovich form has been obtained, let us show that one may apply the theory of stochastic Lie systems to this example. With this aim, one has to recall that the two components of the Stratonovich operator \eqref{Eq:Str} are related to two $t$-dependent vector fields, corresponding to its two components, $\mathfrak{S}_1,\mathfrak{S}_2$, of the form
$$
\mathfrak{S}_1^x\frac{\partial}{\partial x}+\mathfrak{S}_1^y\frac{\partial}{\partial y}=y\frac{\partial}{\partial x}-(\omega_f^2(t)x+[k(t)+\sigma^2(t)/2]y)\frac{\partial}{\partial y},\qquad \mathfrak{S}_2^x\frac{\partial}{\partial x}+\mathfrak{S}_2^y\frac{\partial}{\partial y}=-\sigma(t)y\frac{\partial}{\partial y},
$$
respectively. Note that both $t$-dependent vector fields can be considered as $\mathbb{R}^2$-dependent vector fields that have a trivial dependence on the variable $W_1$ in $(t,W_1)\in \mathbb{R}^2$. 
To describe our model with the theory of stochastic Lie systems, one has to find some finite-dimensional Lie algebra of vector fields on $\T\mathbb{R}^2$, let us say  $V^D$, such that each component of the Stratonovich operator, which can be understood as an $\mathbb{R}^2$-dependent vector field, becomes  a linear combination with coefficients depending on $t,W_1$ of a basis of $V_D$. 

In this case, the sought Lie algebra of vector fields can be obtained by considering the vector fields on $\T \mathbb{R}$ given by
\begin{equation}\label{Eq:BaGL22}
 X_{11}=x\frac{\partial}{\partial x}\,,\qquad X_{12}=y\frac{\partial}{\partial x}\,,\qquad X_{21}=x\frac{\partial}{\partial y}\,,\qquad X_{22}=y\frac{\partial}{\partial y}\,.
\end{equation}
 These vector fields span a four-dimensional Lie algebra of vector fields $V_D$ isomorphic to the general Lie algebra, $\mathfrak{gl}_2$, of $2\times 2$ matrices with real coefficients. In fact, the commutation relations in the basis \eqref{Eq:BaGL22} read
 \begin{align*}
 & [X_{11},X_{12}] = -X_{12}\,, && [X_{11}, X_{21}] = X_{21}\,, && [X_{11}, X_{22}] = 0\,,\\
 & [X_{12}, X_{21}] = X_{22} - X_{11}\,, && [X_{12},X_{22}] = -X_{12}\,,&& [X_{21},X_{22}] = X_{21}\,,
 \end{align*}
and they are equal to the commutation relations between
$$
M_{11}=\left[\begin{matrix}
-1&0\\
0&0
\end{matrix}\right],\,\,M_{12}=\left[\begin{matrix}
0&-1\\
0&0
\end{matrix}\right],\,\,M_{21}=\left[\begin{matrix}
0&0\\
-1&0
\end{matrix}\right],\,\,M_{22}=\left[\begin{matrix}
0&0\\
0&-1
\end{matrix}\right].
$$
Hence, the Stratonovich operator \eqref{Eq:Str} describing our model can be written as 
$$
\mathfrak{S}(t, W_1, x, y)=\left(X_{12}-\omega_f^2(t) X_{21}-(k(t)+\sigma^2(t)/2)X_{22},-\sigma(t)X_{22}\right). 
$$
Thus, each component of the Stratonovich operator can be written as a linear combination with coefficients depending on $t,W_1$ (in particular the coefficients of our model only depend on $t$), of the vector fields of a basis of $V_D$. Hence, \eqref{eq:HOWN} is a stochastic Lie system and $V_D$ becomes an associated Vessiot--Guldberg Lie algebra. 

Note that more general models can be considered as stochastic Lie systems by assuming that their Stratonovich operators take the form
$$
\mathfrak{S}(t,W_1,x,y)=\left(\sum_{\alpha,\beta=1}^2b_{\alpha\beta}(t,W_1)X_{\alpha\beta},\sum_{\alpha,\beta=1}^2b^W_{\alpha\beta}(t,W_1)X_{\alpha\beta}\right),
$$
for arbitrary functions $b_{\alpha\beta},b^W_{\alpha\beta}:\mathbb{R}^2\rightarrow \mathbb{R}$, which also admit a Vessiot--Guldberg Lie algebra $V_D$. One could choose even larger Vessiot--Guldberg Lie algebras, e.g. the Lie algebra of affine vector fields on $\mathbb{R}^2$. Moreover, if the functions $k(t)$, $\omega_f(t)$ and $\sigma(t)$ take particular values, e.g. constant ones, it may happen that one could choose a smaller Vessiot--Guldberg Lie algebra. For instance, if $k(t)=\omega_f(t)=\sigma(t)=0$, the Stratonovich operator could be described via a Vessiot--Guldberg Lie algebra spanned by $X_{12}$. Depending on the dimension of the Vessiot--Guldberg Lie algebra, the superposition rule may depend on less or more particular solutions, making numerical methods simpler or more difficult to be applied. This fact will be exaplained in detail after the stochastic Lie theorem in Section \ref{Sec:SSRandLT}.
\demo
\end{example}

\begin{example}\label{ExSLS-SIS} It is worth noting that an It\^{o} stochastic differential equation $\delta\Gamma=\mathfrak{S}(B,\Gamma)\delta B$ with $\mathfrak{S}(B,\Gamma)$ taking the form \eqref{Eq:STL} may not be a stochastic Lie system. This shows that stochastic Lie systems, in the It\^o framework, do not match exactly the form given in classical Lie systems. Let us provide an example of this, with practical implications, using the SIS model in Example \ref{Ex:FirstItoStr}.
Let us study a stochastic differential equation $\delta I=\mathcal{I}(t,\mathcal{B},I)\delta B$ related to the deterministic SIS model for particular values  $N=100$, $\beta=1/2$, $\mu=\gamma=0$, which is indeed a deterministic approximation of it. More specifically, consider the It\^{o} stochastic differential equation
\begin{equation}\label{Eq:PseudoSto}
\delta I=(50-I/2)I\delta t+\sigma(100-I)I\delta \mathcal{B}\,,
\end{equation}
for a  $t$-dependent parameter $\sigma=\sigma(t)$ which is not constant. This model generalises the stochastic SIS system studied in \cite[p. 880]{GGHMP11}. Consider the vector fields
 $$
 Y_1=\frac{\partial}{\partial I},\qquad Y_2=I\frac{\partial}{\partial I},\qquad Y_3=I^2\frac{\partial}{\partial I},
 $$
 which span a three-dimensional Lie algebra $V_R$ of vector fields with commutation constants
 $$
 [Y_1,Y_2]=Y_2,\qquad [Y_1,Y_3]=2Y_2,\qquad [Y_2,Y_3]=Y_3.
 $$
 In fact, this Lie algebra is isomorphic to $\mathfrak{sl}_2$, namely the matrix Lie algebra of traceless $2\times 2$ matrices with real coefficients. 
 Moreover, the operator $\mathcal{I}$ is such that its components can be written as linear combinations with $t$-dependent coefficients of vector fields of $V_R$ in the form
 $$
 \mathcal{I}(t,\mathcal{B},I)=(50Y_2-1/2Y_3,100\sigma(t) Y_2-\sigma(t) Y_3).
 $$
Hence, one notes that \eqref{Eq:PseudoSto} looks like a stochastic Lie system, but we have to recall that the form \eqref{Eq:STL} must appear in the Stratonovich form of our stochastic differential equations. Nevertheless, one has that \eqref{Eq:PseudoSto} is a stochastic differential equation in the It\^o form related to the Stratonovich stochastic differential  equation $\delta I=\mathfrak{S}(t,\mathcal{B},I)\circ\delta B$ of the form
\begin{equation}\label{Eq:StrSIS}
\delta I=\left(-\sigma(t)I^3+\left(-\frac 12+150\sigma^2(t)\right)I^2+(50-5000\sigma^2(t))I\right)\delta t+\sigma(t)(100-I)I\circ \delta \mathcal{B}\,.
\end{equation}
The problem is that for different values of $\sigma(t)$, which is not constant by assumption, the first component of $\mathfrak{S}(t,\mathcal{B},I)$, namely
$$
 -\sigma(t)I^3+\left(-\frac 12+150\sigma^2(t)\right)I^2+(50-5000\sigma^2(t))I 
$$
is a $t$-dependent vector field whose values at different values of $t\in \mathbb{R}$ generate the vector space $E=\langle Z_1=I\partial/\partial I, Z_2=I^2\partial/\partial I, Z_3=I^3\partial/\partial I\rangle$, which cannot be described by the Vessiot–Guldberg Lie algebra $\langle Y_1,Y_2,Y_3\rangle$ nor by any other. Indeed, the elements of $E$ cannot be written as a linear combination of elements of a finite-dimensional Lie algebra. In fact, the successive Lie brackets  
$$
[Z_2,Z_3]=I^4\frac{\partial}{\partial I},\qquad[Z_2,[Z_2,Z_3]]=2I^5\frac{\partial}{\partial I},\qquad[Z_2,[Z_2,[Z_2,Z_3]]]=3\cdot 2I^6\frac{\partial}{\partial I},\qquad \ldots
$$
span an infinite family of linearly independent vector fields on $\mathbb{R}$ that must be included in any Lie algebra containing $E$.  
\demo
\end{example}

There are  It\^{o} stochastic differential equations whose coefficients match the form of the coefficients in \eqref{Eq:STL} and they are still stochastic Lie systems. This is due to the fact that the transformation \eqref{Eq:TranItoStratonovich} maps the initial \eqref{Eq:practicalSDE} into a new stochastic differential equation that retains again the condition \eqref{Eq:STL}. Notwithstanding, this is not the general case.

The following result will be of utility for applications.

\begin{definition} We call a {\it $t$-dependent stochastic Riccati differential equation}  the Stratonovich stochastic differential equation on $\mathbb{R}$ with stochastic variables given by the semi-martingales  $\mathcal{B}^2,\ldots,\mathcal{B}^\ell:\mathbb{R}_+\times \Omega\rightarrow \mathbb{R}$ of the form 
\begin{equation}\label{eq:NF_Stochastic_Riccati_Equation}
\delta \Gamma=\left(\sum_{\alpha=0}^2b_\alpha(t)\Gamma^\alpha\right)\delta t+\sum_{\beta=2}^\ell\sum_{\alpha=0}^2 b_{\beta\alpha}(t) \Gamma^\alpha\circ \delta \mathcal{B}^\beta,\qquad \forall \Gamma\in \mathbb{R},\qquad \forall t\in \mathbb{R},
\end{equation}
for arbitrary $t$-dependent functions $b_\alpha(t),b_{\mu\alpha}(t)$ with $\alpha=0,1,2$ and $\mu=2,\ldots,\ell$. Observe that $\Gamma^\alpha$ here denotes the stochastic variable $\Gamma$ raised to the $\alpha$-th power.
\end{definition}
Note that $t$-dependent stochastic Riccati differential equations are stochastic Lie systems. The following proposition is immediate.
\begin{proposition}
An It\^{o} differential equation of the form 
$$
\delta \Gamma=\left(\sum_{\alpha=0}^2b_\alpha(t)\Gamma^\alpha\right)\delta t+\sum_{\beta=2}^\ell\sum_{\alpha=0}^1 b_{\beta\alpha}(t) \Gamma^\alpha \delta \mathcal{B}^\beta ,\qquad \forall \Gamma\in \mathbb{R},\qquad \forall t\in \mathbb{R}
$$
for arbitrary $t$-dependent functions $b_\alpha(t),b_{\mu\alpha}(t)$ with $\alpha=0,1,2$ and $\mu=2,\ldots,\ell$
is also a $t$-dependent  stochastic Riccati differential equation.    
\end{proposition}

 It is worth noting that the theory of Lie systems can be generalised to many different realms \cite{Dissertationes,LS21}. In particular, there is the theory of foliated Lie systems \cite{CGM00}. This suggests the following generalisation.

\begin{definition}[Foliated stochastic Lie system]
    A {\it stochastic foliated Lie system} is a stochastic system of differential equations on $M$ of the form
  \begin{equation}
    \delta \Gamma = \mathfrak{S}(B, \Gamma)\circ  \delta B\,, \label{eq1bis2}
  \end{equation}
  where $B:\R_+ \times \Omega \rightarrow \mathbb{R}^\ell$ is an 
  $\mathbb{R}^\ell$-valued semi-martingale and $\mathfrak{S}(B, \Gamma) : \T_B \mathbb{R}^\ell
  \to \T_\Gamma M$ is a Stratonovich operator such that 
\begin{equation}\label{decomposition with the Ys}
    \mathfrak{S}_j (B, \Gamma) = \sum_{\alpha = 1}^r b_j^\alpha (B,\Gamma) Y_\alpha (\Gamma)\,,\qquad j=1,\ldots,\ell\,,\qquad \forall B\in \mathbb{R}^\ell,\qquad \forall \Gamma\in M,
    \end{equation}
and the vector fields $\{ Y_1, \ldots, Y_r \}$ on $M$ span an $r$-dimensional Vessiot--Guldberg Lie algebra such that
$b^\alpha_j(B,\Gamma)$, with $\alpha=1,\ldots,r$ and $j=1,\ldots,\ell$, are first integrals of the vector fields $Y_1,\ldots,Y_r$ for every fixed $B\in \mathbb{R}^\ell$. 
\end{definition}

There are many potential applications concerning the generalisation to the stochastic realm of famous types of Lie systems such as matrix, projective Riccati equations, Bernouilli equations, and so on. As illustrated in this work, stochastic Lie systems have  applications in predator-prey models, oscillator type models, et cetera \cite{Ar03,HH11,Sa70}.  Moreover, nonlinear stochastic differential equations are difficult to study. Notwithstanding, under certain conditions, their linearisation can describe their equilibrium properties \cite{Ar03}. Linear or even affine stochastic differential equations with stochastic variables related to semi-martingales are stochastic Lie systems (cf. \cite{LO09}).

\section{Superposition rules and stochastic Lie systems}\label{Sec:SSRandLT}

Let us study the superposition rule notion for stochastic differential equations and the characterisation of stochastic differential equations admitting a superposition rule. This will lead us to review and slightly correct some mistakes in the main theorem in \cite{LO09}. 

\begin{definition}\label{def 1}
    A {\it superposition rule} for a Stratonovich stochastic differential equation of the form~\eqref{stochastic differential equation expression} on a manifold $M$ is a function $\Phi: M^{m + 1} \to M$ such that, for a generic set $\Gamma_1, \ldots, \Gamma_m:\mathbb{R}_+\times \Omega\rightarrow M$, of particular solutions of~\eqref{stochastic differential equation expression}, the general solution $\Gamma$ to~\eqref{stochastic differential equation expression} takes the form
    \[ \Gamma = \Phi (z
    ; \Gamma_1, \ldots, \Gamma_m), \]
    where $z\in M$ is a point to be related to initial conditions. 
\end{definition}

It is remarkable that superposition rules for stochastic differential equations do not depend on  $\mathbb{R}^\ell$. Let us introduce now several concepts that will be useful to describe and calculate superposition rules for stochastic Lie systems.

The {\it diagonal prolongation to $M^k$ of a vector bundle} $\tau:F\to M$  is the vector bundle $\tau^{[k]}:F^k = F\times\overset{(k)}{\dotsb}\times F\mapsto M^k= M\times \overset{(k)}{\dotsb} \times M$, of the form
\begin{equation*}
   \tau^{[k]}(f_{(1)},\dotsc,f_{(k)}) = (\tau(f_{(1)}),\dotsc,\tau(f_{(k)}))\,,\qquad \forall f_{(1)},\ldots,f_{(k)}\in F\,, 
\end{equation*}
with fibers of the form
\begin{equation}\label{Eq:Fiber}    F^k_{(x_{(1)},\ldots,x_{(k)})}=F_{x_{(1)}}\oplus\cdots\oplus F_{x_{(k)}}\,,\qquad \forall (x_{(1)},\ldots,x_{(k)})\in M^k\,.
\end{equation}

Every section $e:M\to F$ of the vector bundle $\tau$ has a natural {\it  diagonal prolongation} to a section $e^{[k]}$ of the vector bundle $\tau^{[k]}$ given by
\begin{equation*}
    e^{[k]}(x_{(1)},\ldots,x_{(k)})=(e(x_{(1)}),\cdots ,e(x_{(k)}))\,,\qquad \forall (x_{(1)},\ldots,x_{(k)})\in M^k.
\end{equation*}
If we consider that every $e(x_{(a)})$ takes values in the $a$-th copy of $F$ within \eqref{Eq:Fiber}, one can write
\begin{equation*}     e^{[k]}(x_{(1)},\ldots,x_{(k)})=e(x_{(1)})+\cdots +e(x_{(k)})\,,\qquad \forall (x_{(1)},\ldots,x_{(k)})\in M^k,
\end{equation*}
which is a simple useful notation for applications. The {\it  diagonal prolongation of a function $f\in \Cinfty (M)$} to $M^k$ is the function on $M^k$ given by 
$$f^{[k]}(x_{(1)},\ldots,x_{(k)})= f(x_{(1)})+\ldots+f(x_{(k)})\,.
$$
Consider also the sections $e^{(j)}$ of $\tau^{[k]}$, where $j\in \{1,\ldots,k\}$ and $e$ is a section of $\tau$, given by
\begin{equation}\label{prol1}
e^{(j)}(x_{(1)},\dots,x_{(k)})=0+\cdots +e(x_{(j)})+\cdots+0\,,\qquad \forall (x_{(1)},\ldots,x_{(k)})\in M^k\,.
\end{equation}
If $\{e_1,\ldots, e_r\}$ is a basis of local sections of the vector bundle $\tau$, then $e_i^{(j)}$, with $j = 1,\dotsc,k$ and $i = 1,\dotsc,r$, is a basis of local sections of $\tau^{[k]}$. For simplicity, we will frequently write $e(x_{(j)})$ instead of $e^{(j)}$ if it is clear what we mean.

Due to the obvious canonical isomorphisms
$$ (\T M)^{[k]}\simeq \T M^k\quad\text{and}\quad (\cT M)^{[k]}\simeq \cT M^k\,, $$
the diagonal prolongation $X^{[k]}$ of a vector field $X\in\X(M)$ can be understood as a vector field $ {X}^{[k]}$ on $M^k$, and the diagonal prolongation, $\alpha^{[k]}$, of a one-form $\alpha$ on $M$ can be understood as a one-form $ {\alpha}^{[k]}$ on $M^k$.  

More explicitly, let $Y$ be a vector field on a manifold $M$. The {\it diagonal prolongation} of $Y$ to $M^k$ is the vector field  
\[ 
    {Y}^{[k]}(x_{(1)},\ldots,x_{(k)})=\sum_{a=1}^kY(x_{(a)})
\]
on $M^k$ obtained by considering $\T M^k
\simeq \T M\times \dotsb \times \T M$ ($k$ times). Even more particularly, if $Y=x\frac{\partial}{\partial x} $ is a vector field on $\mathbb{R}$, then $Y^{[k]}=\sum_{a=1}^kx_{(a)}\frac{\partial}{\partial x_{(a)}} $.

The diagonal prolongation of a vector field  can be extended to $t$-dependent vector fields on $M$, namely mappings $X\colon\mathbb{R}\times M\rightarrow \T M$ such that $X(t,\cdot)$ is a standard vector field on $M$, by assuming that the diagonal extension to $M^k$ of the $t$-dependent vector field $X$ on $M$ is the $t$-dependent vector field $\widetilde{X}^{[k]}$ on $M^k$ whose value for every fixed valued of $t$, let us say $\widetilde{X}^{[k]}_t$, is the diagonal prolongation to $M^k$ of the vector field $X_t$. The space of diagonal prolongations of vector fields in $\mathfrak{X} (M)$ to $M^k$ form a Lie subalgebra of $\mathfrak{X}
(M^k)$. In fact, the mapping $X\in \X(M)\mapsto {X}^{[k]}\in \X(M^{k})$ is a Lie algebra morphism, i.e. it is a linear mapping such that
\begin{equation}
  \label{eq 14} [Y_1,Y_2]^{[k]}=\left[{Y}^{[k]}_1, {Y}^{[k]}_2\right],\qquad \forall Y_1,Y_2\in \mathfrak{X}(M)\,.
  \end{equation}

Let us solve a mistake in the proof of the stochastic Lie theorem in \cite{LO09}. The issue appears in the direct part of the statement \cite[p. 215, Theorem 3.1]{LO09}. In particular, this makes \cite[Remark 3.3.(i)]{LO09} incorrect. To start with, just recall a couple of notions on differential geometry to make our presentation more accessible for researchers working on stochastic differential equations or applications. 

A {\it generalised distribution} $D$ on a manifold $M$ is a correspondence mapping every point $p\in M$ to a subspace $D_p\subset T_pM$. A generalised distribution $D$ is {\it smooth} if, around every point $p\in M$, there exists an open subset $U\ni p$ and vector fields $X_1,\ldots,X_s$ on $U$ such that $D_{p'}=\langle X_1(p'),\ldots,X_s(p')\rangle$ for every $p'\in U$. It is worth noting that the number $s$ may depend on the point $p$. A generalised distribution $D$ is {\it involutive} if for every pair of vector fields $X_1,X_2$ on an open subset $U\subset M$ such that $X_1(p),X_2(p)\in D_p$ for every $p\in U$, we have that $[X_1,X_2](p)\in D_p$ at every $p\in U$. In other words, a generalised distribution is involutive if the Lie bracket of vector fields taking values in the distribution takes values in the distribution too. More practically, it can be proved that a generalised distribution $D$ on $M$ is involutive if for every point $p\in M$ there exists a family of vector fields $X_1,\ldots,X_s$ spanning the distribution on an open neighbourhood $U$ of  $p$, namely $D_{p'}=\langle X_1(p'),\ldots,X_s(p')\rangle$ for $p'\in U$, satisfy that the Lie brackets $[X_i,X_j]$, with $1\leq i<j\leq s$, also take values in $D$ on $U$, namely 
$[X_i,X_j](p')\in D_{p'}$ for every $p'\in U$ and $1\leq i<j\leq s$.
In this work, we have said `distributions' instead of  `generalised distributions' to simplify our terminology, as commonly done in the literature. Moreover, all inspected distributions  are smooth.

For the sake of completeness, we will state that part of the work adopting our notation and writing in full the contents referenced by the labels used in \cite[Theorem 3.1]{LO09}. We hereafter refer to
\begin{quote}\it
    ``Moreover, the fact that one has the general property
    $$
    \left[ {Z}^{[k+1]}_1,(Z_2+\lambda Z_3)^{[k+1]}\right]=([Z_1,Z_2+\lambda Z_3])^{[k+1]},\qquad \forall Z_1,Z_2,Z_3\in \mathfrak{X}(M),\qquad \lambda\in \mathbb{R},
    $$
    and the hypothesis on $\{Y_1,\dotsc, Y_\kappa\}$ that 
    they span an involutive distribution\footnote{In \cite[Theorem~3.1]{LO09}, it is stated that $\{Y_1,\ldots,Y_\kappa\}$ span an involutive distribution, i.e. $[Y_i,Y_j]=\sum_{k=1}^\kappa f_{ijk}Y_k$ for suitable functions $f_{ijk}$ with $i,j,k=1,\ldots,r$. The proof also notes that $\{Y_1,\ldots,Y_r\}$ are `in involution,' again meaning they span an involutive distribution. This terminology may be misleading, as in some references `vector fields in involution' is used to mean that the vector fields commute.} imply by the classical Fr\"obenius theorem that 
    $$
    D=\left\langle {Y}^{[k+1]}_1,\ldots,{Y}^{[k+1]}_\kappa\right\rangle
    $$
    is integrable.''
\end{quote}
which is incorrect. Let us explain the mistake and provide a counterexample. Our counterexample will be relevant because it will show that certain stochastic differential equations in the It\^{o} approach are not stochastic Lie systems and do not admit a superposition rule.  

The problem relies on the fact that if $Y_1,\dotsc,Y_\kappa$ span an involutive distribution on $M$, then the diagonal prolongations $\widetilde Y^{[k+1]}_1,\ldots,\widetilde Y^{[k+1]}_\kappa$ do not need to span the distribution given by the diagonal prolongations of the vector fields taking values in the  distribution spanned by $Y_1,\ldots,Y_\kappa$. In fact, as shown next, the distribution spanned by $\widetilde Y^{[k+1]}_1,\ldots,\widetilde Y^{[k+1]}_\kappa$ need not be involutive at all.

For instance, consider the two vector fields on $\mathbb{R}_\circ=\mathbb{R}\setminus\{0\}$ given by
$$
Y_1=x^2\frac{\partial}{\partial x}\,,\qquad Y_2=x^3\frac{\partial}{\partial x}
$$
that span an involutive distribution on $\mathbb{R}_\circ=\{x\in \mathbb{R}|x\neq 0\}$. Indeed, both vector fields span a distribution $\mathcal{D}_x=\langle x^2\partial/\partial x,x^3\partial/\partial\rangle=\T_x\mathbb{R}_\circ$ for $x\in \mathbb{R}_\circ$ and every two vector fields taking values in $\T\mathbb{R}_\circ$ have a Lie bracket contained in  $\T\mathbb{R}_\circ$.   
Notwithstanding, their diagonal prolongations, ${Y}^{[s]}_1, {Y}^{[s]}_2$, to $\mathbb{R}_\circ^s$ with $s>2$ do not need to span an involutive distribution. Indeed, one has the diagonal prolongations
$$
 {Y}_1^{[s]}=\sum_{a=1}^sx_{(a)}^2\frac{\partial}{\partial x_{(a)}},\qquad  {Y}_2^{[s]}=\sum_{a=1}^s x_{(a)}^3\frac{\partial}{\partial x_{(a)}}
$$
on $(\mathbb{R}_\circ)^s$. Their successive commutators become, as stated in \cite{LO09}, diagonal prolongations of an element of the involutive distribution spanned by $Y_1,Y_2$, namely $\T \mathbb{R}_\circ$. In particular,
\[
\ad_{ {Y}^{[s]}_1}^k {Y}^{[s]}_2=\stackrel{k-{\rm times}}{\overbrace{\left[ {Y}_1^{[s]},[
\ldots,[\ldots,[\ldots, [ {Y}_1^{[s]}\right.}}\left., {Y}_2^{[s]}]\ldots]\ldots]\ldots]\right]=
\sum_{a=1}^sk!x_{(a)}^{3+k}\frac{\partial}{\partial x_{(a)}}=\left(k!x^{k+3}\frac{\partial}{\partial x}\right)^{[s]},
\]
for $k=1,2,\ldots$
Notwithstanding,  $\ad_{ {Y}^{[s]}_1}^k {Y}^{[s]}_2$, with $k\in \{2,3,4,\ldots\}$, does not take values  in the distribution $\mathcal{D}=\left\langle  {Y}^{[s]}_1, {Y}^{[s]}_2\right\rangle$ at a generic point of $\mathbb{R}_\circ^s$. Even worse, the smallest  (in the sense of inclusion) involutive distribution containing $ {Y}^{[s]}_1, {Y}^{[s]}_2$ spans the whole tangent space to $\mathbb{R}_\circ^s$ at almost every point and a superposition rule for a system described by a generic combination of $Y_1,Y_2$ does not exist as it must be constructed from the non-constant first integrals of the vector fields taking values in an involutive distribution containing  $ {Y}^{[s]}_1, {Y}^{[s]}_2$ for some $s>2$ (see the proof for the Lie theorem in \cite{CGM07,Dissertationes,LS21,LO09}).  In fact, the vector fields $ {Y}^{[s]}_{\mu}=\sum_{a=1}^sx_{(a)}^{\mu+1}\frac{\partial}{\partial x_{(a)}}$, with $\mu=1,\ldots,s$, on $\mathbb{R}_\circ^s$ are linearly independent almost everywhere. To verify this fact, it is enough to see that the determinant of their coefficients in the basis $\frac{\partial}{\partial x_{(a)}}$, with $a=1,\ldots,s$, read
$$
    \begin{vmatrix}
        x_{(1)}^2 & x_{(2)}^2 & \cdots & x_{(s)}^2 \\
        x_{(1)}^3 & x_{(2)}^3 & \cdots & x_{(s)}^3 \\
        \vdots & \vdots & & \vdots \\
        x_{(1)}^{s+1} & x_{(2)}^{s+1} & \cdots & x_{(s)}^{s+1}
    \end{vmatrix} = x_{(1)}^2\dotsb x_{(s)}^2 \begin{vmatrix}
        1 & 1 & \cdots & 1 \\
        x_{(1)} & x_{(2)} & \cdots & x_{(s)} \\
        \vdots & \vdots & & \vdots \\
        x_{(1)}^{s-1} & x_{(2)}^{s-1} & \cdots & x_{(s)}^{s-1}
    \end{vmatrix} = \prod_{a = 1}^s x_{(a)}^2 \prod_{1\leq i < j \leq s} (x_{(j)} - x_{(i)})\,.
$$
which causes the smallest involutive generalised distribution containing $ {Y}^{[s]}_1, {Y}^{[s]}_2$ on $\mathbb{R}_\circ^s$ to be equal to $\T \mathbb{R}_\circ^s$ almost everywhere. This makes the existence of common non-constant first integrals for $ {Y}^{[s]}_1, {Y}^{[s]}_2$, which will be common non-constant first integrals for all $Y^{[s]}_1,Y^{[s]}_2,\ldots$  impossible.

The previous counterexample is very important as it concerns the stochastic generalisations of the so-called Abel equations (see \cite{Dissertationes} and references therein). Hence, the mistake in \cite{LO09} has potential practical consequences. Moreover, recall that the SIS model in the It\^{o} form \eqref{Eq:PseudoSto}, for a non-constant function $\sigma(t)$, takes the Stratonovich form \eqref{Eq:StrSIS}. We already showed that if one tries to write the $t$-dependent coefficient of the related Stratonovich operator for $\delta t$ as a linear combination with $t$-dependent functions of a family of vector fields on $\mathbb{R}$ spanning a finite-dimensional Lie algebra, one finds that, for a generic $t$-dependent function $\sigma(t)$, one has to obtain a finite-dimensional Lie algebra of vector fields on $\mathbb{R}$ containing $x^2\frac{\partial}{\partial x},x^3\frac{\partial}{\partial x}$, which is impossible as shown in Example \ref{Ex:FirstItoStr}.

Despite the above mistake, the stochastic Lie--Scheffers theorem in \cite{LO09} can be corrected by assuming that the initial family of vector field $Y_1,\ldots,Y_r$ close an $r$-dimensional Lie algebra. Note that the assumption that $Y_1,\ldots,Y_r$ are linearly independent over $\mathbb{R}$ is necessary. Otherwise, e.g. if $Y_{r-1}=Y_r$, one obtains that the vector fields $Y_1^{[m]},\ldots,Y_r^{[m]} $ are always linearly dependent and Lemma 3.2 of \cite{LO09} cannot be applied in \cite[p. 918]{LO09}. On the other hand, if $Y_1,\ldots,Y_r$ are linearly independent, their diagonal prolongations become linearly independent at a generic point exactly when $m$ is such that they span a distribution of rank $r$ at a generic point (cf. \cite{Dissertationes}). Then, Lemma 3.2 can be applied normally as in \cite[p. 918]{LO09}. 

\begin{theorem}
  [Stochastic Lie theorem]\label{teorema Lie}Let
  \begin{equation}
    \delta \Gamma = \mathfrak{S}(B, \Gamma) \circ\delta B \label{eq1bis}
  \end{equation}
  be a stochastic differential equation on $M$, where $B :
  \mathbb{R}_+ \times \Omega \rightarrow \mathbb{R}^\ell$ is a given
  $\mathbb{R}^\ell$-valued semi-martingale and $\mathfrak{S}(B, \Gamma) : \T_B \mathbb{R}^\ell
  \to \T_\Gamma M$, for every $B
\in \mathbb{R}^\ell$ and $\Gamma\in M$, is a Stratonovich operator from
  $\R^\ell$ to $M$. Then, \eqref{eq1bis} admits a superposition rule if and only if
\begin{equation}\label{decomposition with the Ys2} \mathfrak{S}_j (B, \Gamma) = \sum_{\alpha = 1}^r b_j^\alpha (B)
      Y_\alpha (\Gamma)\,,\qquad j=1,\ldots,\ell\,,
    \end{equation}
    for every  $B\in \mathbb{R}^\ell$ and $\Gamma\in M$, where the vector fields $\{ Y_1, \ldots, Y_r \}$ span an $r$-dimensional Vessiot--Guldberg Lie algebra  on $M$.
\end{theorem}
\begin{proof} Let us prove the part of the direct implication of \cite[Theorem 3.1]{LO09}  that requires some comments in light of our correction. There always exists a number $m$ such that the diagonal prolongation of a basis $Y_1,\ldots,Y_r$ of the Vessiot--Guldberg Lie algebra of the Lie system to $M^m$ reaches rank $r$ at a generic point. It was proved in \cite{Dissertationes} that this happens exactly when the diagonal prolongations span a distribution of rank $r$ almost everywhere. Let be $\mathcal{D}_0$ the distribution  spanned by vector fields $Y^{[m+1]}_1,\ldots,Y^{[m+1]}_r$. Then $\mathcal{D}_0$ has constant rank on an open subset of $M^{m+1}$. Since the Lie brackets $[Y_i,Y_j]$ are linear combinations of the vector fields of $V$, with constant coefficients, then their diagonal prolongations to $M^{m+1}$ are linear combinations, with constant coefficients, of the $Y^{[m+1]}_1,\ldots,Y^{[m+1]}_r$ and $\mathcal{D}_0$ is integrable. Then, one can extend $\mathcal{D}_0$ to an integrable distribution $\mathcal{D}$ of corank $n$ in $M^{m+1}$, i.e. $\mathcal{D}_0\subset \mathcal{D}$ on every point of an open neighbourhood of a point in $M^{m+1}$. This can always be done in such a manner that the leaves of $\mathcal{D}$ project diffeomorphically onto an open subset of $M^m$ via the canonical projection $\pi:(\Gamma_{(1)},\ldots,\Gamma_{(m+1)})\in M^{m+1}\mapsto(\Gamma_{(1)},\ldots,\Gamma_{(m)})\in M^m$. The obtained expression is a local superposition rule for \eqref{eq1bis} as shown in \cite[Theorem 3.1]{LO09} in basis of \cite[Proposition 2.4]{LO09}.
\end{proof}

It is important to highlight that the proof of the stochastic Lie theorem shows, like its non-stochastic analogue, that if the Vessiot--Guldberg Lie algebra has  dimension $r$, then $m$, which standard for the number of particular solutions of the associated superposition rule, must fulfils the condition $m\dim M\geq r$. This requirement is crucial to guarantee that the diagonal prolongations $Y^{[m]}_1,\ldots,Y^{[m]}_r$ are linearly independent at a generic point of $M^m$. Consequently, a larger Vessiot--Guldberg Lie algebra on a manifold $M$ is associated with a greater number of particular solutions for the superposition rule of the related stochastic Lie system, namely $m$,. On the other hand, larger Vessiot--Guldberg Lie algebras lead to superposition rules for larger families of stochastic Lie systems admitting a common superposition rule (see \cite{Dissertationes} for an exploration of these facts in the classic non-stochastic context, which is completely analogous).

\section{Hamiltonian stochastic Lie systems}\label{Sec:Hamiltonian_Stochastic_Lie_Systems}

This section studies stochastic Lie systems admitting a Vessiot--Guldberg Lie algebra of Hamiltonian vector fields relative to a geometric structure. In particular, we mainly focus on Hamiltonian vector fields relative to a symplectic form, but our analysis can be immediately extended to more general Hamiltonian systems, e.g. related to Poisson manifolds. To motivate our approach, let us consider a stochastic differential equation on $\mathbb{R}^2$ of the form
\begin{equation}\label{eq:PG}
\begin{cases}
\dfrac{dx}{d t}=x  ({a}_{1}(t,\mathcal{B})-b_{1} y),\\
\dfrac{dy}{d t}= y  ({a}_{1}(t,\mathcal{B})+b_{2} x),
\end{cases}
\end{equation}
for a function $a_1\in \Cinfty(\mathbb{R}^2)$ and constants $b_1,b_2\in \mathbb{R}$. This system is a stochastic generalisation of a type of Palomba--Goodwin model, which in turn is a particular case  of Lotka--Volterra system \cite{Goo_82,Pal_39}.
 Our stochastic generalisation is obtained by assuming that the standard $t$-dependent coefficients  depend, additionally, on a semi-martingale $\mathcal{B}$. In this case,  the associated Stratonovich operator takes the form
\[
\mathfrak{S}(t,\mathcal{B},x,y)=\begin{pmatrix} x  ({a}_{1}(t,\mathcal{B})-b_{1}y)\\y  ({a}_{1}(t,\mathcal{B})+b_{2}x)\end{pmatrix}\,,
\]
which allows us to write
$$
\delta\Gamma= \begin{pmatrix} \delta x\\\delta y \end{pmatrix} = \mathfrak{S}(t,\mathcal{B},x,y) \delta t\,.
$$
Moreover,  $\mathfrak{S}(t,\mathcal{B},x,y)$ can be written as
$$
\mathfrak{S}(t,\mathcal{B},x,y)=a_1(t,\mathcal{B})X_1+X_2\,,
$$
where the vector fields $X_1$ and $X_2$ are given by
\[
X_1=x\frac{\partial}{\partial x}+y\frac{\partial}{\partial y}\,,\qquad X_2=xy\left(-b_1 \frac{\partial}{\partial x}+b_2 \frac{\partial}{\partial y}\right).
\]
Since $
[X_1,X_2]=-X_2,
$
the vector fields $X_1,X_2$ span a non-Abelian two-dimensional Lie algebra $\mathfrak{h}_2$. Hence, system \eqref{eq:PG} is a stochastic Lie system. 

The vector fields $X_1,X_2$ are Hamiltonian relative to the symplectic structure
$\omega =\frac{1}{xy} \d x\wedge \d y$ on $\mathcal{O}=\{(x,y)\in \mathbb{R}^2:xy\neq 0\}$. In fact,
$$
\iota_{X_1}\omega=\d\ln|y/x|\,,\qquad \iota_{X_2}\omega = \d\left(-b_1y-b_2x\right).
$$
Hence, we call \eqref{eq:PG},  a {\it Hamiltonian stochastic Lie system} on $\mathcal{O}$.  
Moreover, \eqref{eq:PG} could be further generalised, e.g. by considering 
$$
\delta \Gamma=\mathfrak{S}(t,\mathcal{B},x,y)\delta t+\mathfrak{S}'(t,\mathcal{B},x,y)\circ\delta \mathcal{B},
$$
where $\mathfrak{S}'(t,\mathcal{B},x,y)=f_1(t,\mathcal{B})X_1+f_2(t,\mathcal{B})X_2$ for arbitrary functions $f_1,f_2\in \Cinfty(\mathbb{R}^2)$.

Let us analyse the stochastic SIS model in the Stratonovich approach given by \cite{CMZ_19}
\begin{align*}
\delta S&=\left[\Lambda-\mu_1S-
\frac{\beta SI}{\kappa+S}+\gamma I\right]\delta t-\sigma S\circ\delta \mathcal{B}\,,\\
\delta I&=\left[\frac{\beta SI}{\kappa+S}-(\mu_1+\gamma) I\right]\delta t-\sigma I \circ\delta \mathcal{B}\,,
\end{align*}
where $\kappa$ is a rate of disease-related
death, $\Lambda$ is an input of new members, $\mu$ is a natural mortality, $\beta$ is a transmission rate, $\gamma$ is a rate of recovery, $\sigma^2$ is the intensity of white noise, $\mu_1=\mu+\sigma^2/2$, and $\mathcal{B}$ is a Brownian motion. Let us consider the case $\kappa=0$, namely 
\begin{align}
    \delta S &= \left[\Lambda-\mu_1S+(\gamma-\beta) I\right]\delta t - \sigma S\circ\delta \mathcal{B} \notag \\
    \delta I &= \left[ -\mu_1 I-(\gamma-\beta) I\right]\delta t - \sigma I \circ\delta \mathcal{B} 
    \label{Eq:NewSIS}
\end{align}
one can consider the vector fields on $\mathbb{R}^2$ given by
$$
\bar{X}_1=  S\parder{}{S} + I\parder{}{I}\,,\quad \bar{X}_2 = I\parder{}{S} - I\parder{}{I}\,,\quad \bar{X}_3=\parder{}{S}.
$$
Then,
$$
[\bar X_1,\bar X_3]=-\bar X_3\,,\qquad [\bar X_2,\bar X_3]=0,\qquad [\bar X_1,\bar X_2]=0.
$$
Hence, the system becomes a stochastic Lie system. In fact, the Stratonovich operator reads
\begin{equation}
    \mathfrak{S}(B,\Gamma) = (\Lambda \bar X_3-\mu_1\bar X_1+(\gamma-\beta)\bar X_2, -\sigma \bar X_1)\,,
\end{equation}
where we consider solutions  $\Gamma:(t,\omega_p)\in\mathbb{R}_+\times \Omega\mapsto (S,I)\in \mathcal{U}=\{(S,I)\in \mathbb{R}^2:SI\neq 0\}$. In the particular case of $\Lambda=0$, one can consider the locally defined symplectic form, away from points with $(S+I)I\neq 0$, of the form
$$
\omega=\d \xi_1\wedge \d \xi_2
$$
where $\xi_1,\xi_2$ are local coordinates such that $\bar X_1=\frac{\partial}{\partial \xi_1}$ and $\bar X_2=\frac{\partial}{\partial \xi_2}$. The existence  of $\xi_1,\xi_2$ is due to the fact that $\bar X_1\wedge \bar X_2$ does not vanish and $[\bar X_1,\bar X_2]=0$ away from points with $(S+I)I\neq 0$. Note that system \eqref{Eq:NewSIS} can be generalised to consider $t$-dependent coefficients. 

Let us consider the stochastic differential equation on  $\mathbb{R}^n$ with a Wiener process $W_1$ (see \cite{HH11,Ha19,YFBL_17} for analogues on $\mathbb{R}$) in Stratonovich form given by
\begin{equation}\label{eq:HHOWN}
\delta\begin{pmatrix}
    x_i \\ y_i
\end{pmatrix} = \begin{pmatrix}
    \lambda(t) & a(t) \\ b(t) & -\lambda(t)
\end{pmatrix} \begin{pmatrix}
    x_i \\ y_i
\end{pmatrix} \delta t + \begin{pmatrix}
    \lambda'(t) & a'(t) \\ b'(t) & -\lambda'(t)
\end{pmatrix} \begin{pmatrix}
    x_i \\ y_i
\end{pmatrix} \circ \delta W_1\,,\qquad i=1,\ldots,n\,,
\end{equation}
 which is defined on $\cT\mathbb{R}^n$. This recovers subcases of Example \ref{eq:WhiteOscil}, namely equation \eqref{eq:HOWN} with  $k(t)=\sigma(t)=0$. It can be proved that \eqref{eq:HOWN} is not Hamiltonian relative to any symplectic form on $\mathbb{R}^2$ for arbitrary $t$-dependent coefficients as it leads to a Vessiot--Guldberg Lie algebra on the plane related to $\mathfrak{gl}_2$  \cite{LS21}. But \eqref{eq:HHOWN} can now be related to a Vessiot--Guldberg Lie algebra $V$  on $\T^*\mathbb{R}^n$ spanned by
 \begin{equation}\label{eq:StoHarOsc}
 X_1=\sum_{i=1}^n y_i\frac{\partial}{\partial x_i}
\,,\qquad X_2=\frac 12 \sum_{i=1}^n\left(x_i\frac{\partial}{\partial x_i}-y_i\frac{\partial}{\partial y_i}\right)\,,\qquad X_{3}=-\sum_{i=1}^n x_i\frac{\partial}{\partial y_i}
\,.
 \end{equation}
 Since,
 $$
 [X_1,X_2]=X_1\,,\qquad [X_1,X_3]=X_2\,,\qquad [X_2,X_3]=X_3\,,
 $$
 this Lie algebra is isomorphic to $\mathfrak{sl}_2$ and it is the diagonal prolongation to $(\T^*\mathbb{R})^n$  of a Vessiot--Guldberg Lie algebra isomorphic to $\mathfrak{sl}_2$ on $\T^*\mathbb{R}$ (see \cite{LS21}). Moreover, it is known to be a Lie algebra of Hamiltonian vector fields relative to the symplectic form
 $$
 \omega=\sum_{i=1}^n\d x_i\wedge \d y_i\,.
 $$
Hence, $X_1,X_2,X_3$ have Hamiltonian functions given by
 \begin{equation}\label{Eq:StoOscHamFun}
 h_1 = \sum_{i=1}^n \frac{y_i^2}{2}\,,\qquad h_2 = \frac{1}{2}\sum_{i=1}^nx_iy_i\,,\qquad h_3=\sum_{i=1}^n \frac{x_i^2}{2}\,.
 \end{equation}
 relative to $\omega$ that span a Lie algebra isomorphic to $\mathfrak{sl}_2$. 
 
Now, \eqref{eq:HHOWN} can be related to  Stratonovich operator of the form
 $$
\mathfrak{S}(x ,y ,t,W_1)=\left(\lambda(t) X_2+a(t)X_1-b(t)X_3,\lambda'(t)X_2+a'(t)X_1-b'(t)X_3\right), $$
which turns \eqref{eq:HHOWN} into a  Hamiltonian stochastic Lie system and $V$ into an associated Vessiot--Guldberg Lie algebra isomorphic to $\mathfrak{sl}_2$.

\begin{definition}
    A {\it Hamiltonian stochastic Lie system} on a manifold $M$ relative to a probability space $\Omega$ and a semi-martingale  $B\colon \mathbb{R}_+\times\Omega\rightarrow \mathbb{R}^\ell$ is a stochastic Lie system related to a Stratonovich operator $\mathfrak{H}$ admitting an associated Vessiot--Guldberg Lie algebra on $M$ consisting of Hamiltonian vector fields relative to some geometric structure on $M$.
\end{definition}

Of course, the above definition means that there are Hamiltonian stochastic Lie systems related to symplectic, Poisson, Jacobi, contact or Dirac geometries, among others. It is worth noting that we are mainly interested in Hamiltonian stochastic Lie systems given by
\begin{equation}\label{eq:HamStoLieSys}
\delta \Gamma=\sum_{\alpha=1}^rb_1^\alpha(t)X_\alpha \delta t+\sum_{i=2}^\ell\sum_{\alpha=1}^rb_i^\alpha(t)X_\alpha\circ\delta \mathcal{B}^i
\end{equation}
such that the vector fields $X_1,\ldots,X_r$ span an $r$-dimensional Vessiot--Guldberg Lie algebra of Hamiltonian vector fields on $M$ relative to a symplectic form. Note that the coefficients in \eqref{eq:HamStoLieSys} depend only on time, although dependence on $\mathcal{B}^2,\dotsc,\mathcal{B}^\ell$ will be also analysed. In our main case of study, \eqref{eq:HamStoLieSys} is then related to an $\ell$-family of $t$-dependent Hamiltonian functions being each of them a linear combination with coefficients depending on $t$ of certain Hamiltonian functions $h_1,\dotsc,h_r\in \Cinfty(M)$ contained in a finite-dimensional Lie algebra of Hamiltonian functions relative to the Poisson bracket of the symplectic manifold. It may happen that the Hamiltonian functions $h_1,\ldots,h_r$ need to be enlarged with an additional constant function to close on a Lie algebra, but this option is hereafter skipped for simplicity. In particular, the $\ell$-family of $t$-dependent vector fields in \eqref{eq:HamStoLieSys} have $t$-dependent Hamiltonian functions
$$
\widetilde{h}_1=\sum_{\alpha=1}^rb^\alpha_1(t)h_\alpha\,,\qquad \widetilde{h}_j=\sum_{\alpha=1}^rb^\alpha_j(t)h_\alpha\,,\qquad j=2,\ldots,\ell\,,
$$
respectively. 
It is standard to call $\widetilde{h}_1$ the {\it Hamiltonian of the stochastic differential equation}. We call $\langle  h_1,\ldots,h_r\rangle$ a {\it  Lie--Hamilton Lie algebra} of a Hamiltonian stochastic Lie system.

\begin{figure}
    \centering
    \includegraphics[width=0.5\linewidth]{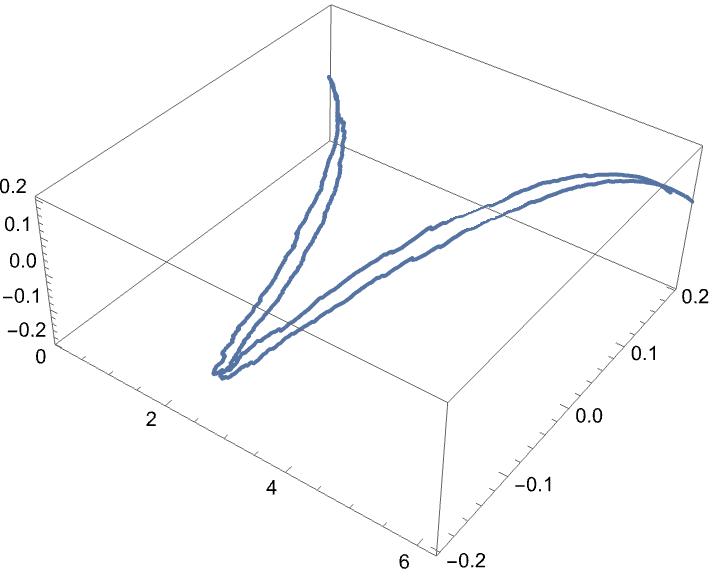}\includegraphics[width=0.5\linewidth]{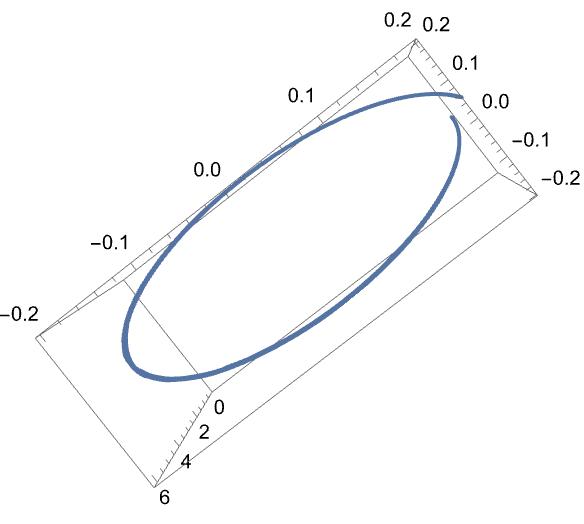}
    \caption{These are two representations of the evolution in terms of the time of a particular solution to the Hamiltonian Lie system $\delta y=-x\delta t-Ax\delta W,\delta x=y\delta t+Ay\delta W$, with initial condition $(.2,0)$ a semi-martingale $W$, and a parameter $A$. The symplectic structure is $\omega=\d x\wedge \d y$. The system has a strong constant of motion $x^2+y^2$. The constant of motion is always conserved, but solutions jump back and forward relative to the deterministic solution.}
    \label{fig:strong-constant}
\end{figure}
\section{Stability for stochastic Hamiltonian systems}\label{Sec:SHS}
Let us recall the stability theory for stochastic differential equations and apply it to stochastic Hamiltonian systems. In particular, we will be specially interested in Hamiltonian stochastic Lie systems. For a survey on the theory for general stability of stochastic differential equations and other related results see \cite{Ar74,LO08,Sa70}.

Consider a stochastic differential equation on $M$ of the particular form \begin{equation}
  \label{stochastic} \delta \Gamma =\mathfrak{S}_t (t,
  \Gamma)\delta t+ \mathfrak{S}_\mathcal{B} (t,
  \Gamma)\circ \delta \mathcal{B}\,,\qquad \forall \Gamma\in  M,\qquad \forall t\in \mathbb{R},
\end{equation}
where $B=(t,\mathcal{B})\colon\mathbb{R}_+ \times \Omega \rightarrow \R^\ell$ is an
$\R^\ell$-valued semi-martingale and 
$$\mathfrak{S} (t, \Gamma)=(\mathfrak{S}_t(t,\Gamma),\mathfrak{S}_\mathcal{B}(t,\Gamma)) \colon \T_B \mathbb{R}^\ell
\to \T_\Gamma M$$ is the associated Stratonovich operator. As in previous sections, it is assumed that initial conditions are values in $M$ chosen with probability equal to one. Note that coefficients of the Stratonovich operator are considered to depend only on $t$ and $\Gamma$. Stochastic differential equations of this type are common in the literature \cite{BCC_01}.

An {\it equilibrium point} for the stochastic differential equation \eqref{stochastic} is a point $\Gamma_e\in M$ such that
$$
\mathfrak{S} (t,
  \Gamma_e)=0\,,\qquad \forall t\in \mathbb{R}\,.
$$
In this case, the stochastic disturbance, which is described by $\mathfrak{S}_\mathcal{B}$, does not act at the equilibrium point $\Gamma_e$.  
Note that if $\mathfrak{S}_\mathcal{B}$ is not assumed to be zero at the equilibrium point for every $t\in \mathbb{R}$, solutions may move away from that point depending on the values of the stochastic variable (see Figure \ref{fig:enter-label}).

\begin{figure}
    \centering
    \includegraphics[width=0.5\linewidth]{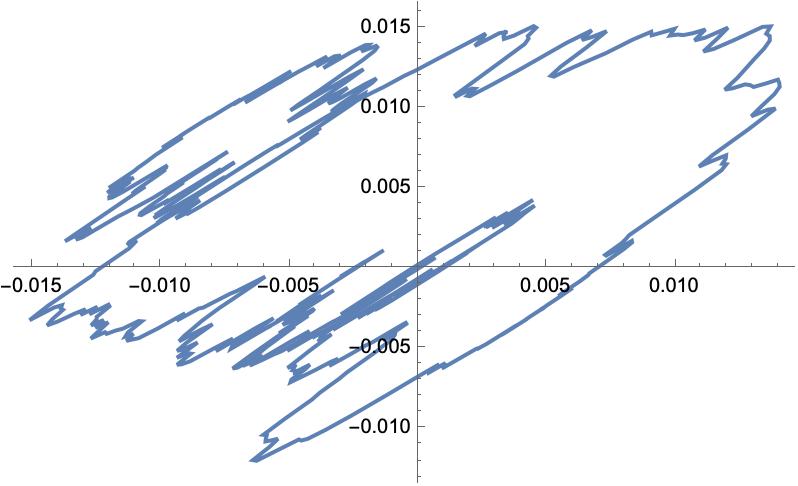}
    \caption{Representation of a solution to $\delta x=y\delta t-0.008\delta W,\delta y=-x\delta t-0.008\delta W$  on $\mathbb{R}^2$ with initial condition $(0,0)$ and stochastic variable given by a semi-martingale $W$, showing that a non-vanishing stochastic part of the stochastic differential equation may move solutions away from a deterministic equilibrium point.}
    \label{fig:enter-label}
\end{figure}

\begin{definition}
Consider the Stratonovich differential equation 
\begin{equation}\label{eq:Stabil}
\delta \Gamma = \mathfrak{S}(t, \Gamma)\circ\delta B
\end{equation}
with solutions $\Gamma\colon\mathbb{R}_+\times\Omega\rightarrow M$.
Given $\Gamma_0 \in  M$ and $s \in  \mathbb{R}$, denote by $\Gamma^{s,\Gamma_0}$ the particular solution of \eqref{eq:Stabil} such that $\Gamma^{s,\Gamma_0}(\omega_0)=\Gamma_0$ for all $\omega_0\in \Omega$. Let $\Gamma_e \in  M$ be an equilibrium point of \eqref{eq:Stabil}. 
The equilibrium point $\Gamma_e$ is
{\it almost surely (Lyapunov) stable} if for  any open neighbourhood $U$ of $\Gamma_e$ there exists a
neighbourhood $\mathcal{U}$ of $\Gamma_e$ such that for any $\Gamma'\in \mathcal{U}$, one has that $\Gamma^{s,\Gamma'}\subset U$ almost surely (a.s.).
 
\end{definition}

Let us provide a relevant analogue of the method used in Hamiltonian symplectic systems to verify stability \cite{AM78}. Before that, let us provide a definition of strongly conserved quantities for Hamiltonian stochastic differential equations (see \cite[Definition 2.2]{LO08}). 
\begin{definition}A function $f\in \Cinfty(M)$ is said to be {\it strongly  conserved}
 of a stochastic Hamiltonian \eqref{eq:Stabil} if, for any particular solution $\Gamma$ with initial condition $\Gamma_0$,
  we have that $f(\Gamma)=f(\Gamma_0)$.
\end{definition}

Strongly conserved quantities can also be defined for systems \eqref{eq:Stabil} whose Stratonovich operator also depend on $B$.  The most relevant result for our purposes in the study of stability of Hamiltonian stochastic Lie systems is the following proposition (see \cite[Theorem 2.15]{LO08} and references therein). 
\begin{proposition} (Stochastic Dirichlet’s Criterion) Assume that there exists a function $f\in \Cinfty(M)$ such that $\d f_{\Gamma_e} = 0$ and that the
quadratic form $\d^2f_{\Gamma_e}$ is positive or negative definite. If $f$ is a strongly conserved quantity on the solutions of \eqref{eq:Stabil}, then the equilibrium point $\Gamma_e$ is almost surely stable.
\end{proposition}

Let us turn now to Hamiltonian stochastic differential equations on symplectic manifolds. In this case, we focus on stochastic differential equations of the type
\begin{equation}\label{eq:HamSDE}
\delta  \Gamma=\mathfrak{H}(t,\Gamma )\circ \delta B\,,
\end{equation}
where $B\colon\mathbb{R}_+\times \Omega\rightarrow \mathbb{R}^\ell$ is a semi-martingale and $\mathfrak{H}(t,\Gamma):\T_B \mathbb{R}^\ell\rightarrow \T_\Gamma M$ is a Stratonovich operator such that
$$
\mathfrak{H}(t,\Gamma)=(Y_1(t,\Gamma),\ldots,Y_\ell(t,\Gamma))\,, 
$$
where $Y_1,\ldots,Y_\ell$ stand for $t$-dependent Hamiltonian vector fields relative to a symplectic form $\omega$ on $M$ with $t$-dependent Hamiltonian functions $\widetilde{h}_1,\ldots,\widetilde{h}_\ell\in \Cinfty(\mathbb{R}\times M)$, respectively. The symplectic structure in this case allows the study of the system via its Hamiltonian functions, which provides powerful methods to study their properties. Note that \eqref{eq:HamSDE} has indeed  an associated function $\widetilde{h}\colon\mathbb{R}\times M\rightarrow \mathbb{R}^{\ell}$ of the form
$$
\widetilde{h}=(\widetilde{h}_1,\ldots,\widetilde{h}_\ell).
$$
Recall that $\widetilde{h}_1$ is normally called the {\it Hamiltonian of the system}, but Hamiltonian stochastic differential equations have several associated Hamiltonians. In fact,  $\widetilde{h}_1$ need not be conserved even when it is $t$-independent. It is immediate that it is not a constant of motion in general as $\{\widetilde{h}_1,\ldots,
\widetilde{h}_\ell\}$ not need to be zero in general. Notwithstanding, it is remarkable that an $f\in \Cinfty(M)$ is a strong constant of motion for \eqref{eq:HamStoLieSys} if  (cf. \cite{LO08})
$$
\{(\widetilde{h}_\alpha)_t,f\}=0\,,\qquad \alpha=1,\ldots,\ell\,,\qquad \forall t\in \mathbb{R}.
$$

Let us assume that there exists a function $f$ that is a constant of motion of $\delta\Gamma$ and it also satisfies that it has a strict minimum at an equilibrium point. Then, it is immediate that $\delta\Gamma$ has a stable equilibrium point.

It is immediate to apply the above results to linear or affine stochastic differential equations, which are stochastic Lie systems, and to choose cases with a Vessiot--Guldberg Lie algebra of Hamiltonian vector fields relative to a symplectic structure.

Let us consider a particular case of the the stochastic differential equation on  $\mathbb{R}^n$ with a Wiener process $W_1$ given in \eqref{eq:HHOWN} for $2a(t)=a'(t)=2$, $2b(t)=b'(t)=-2\omega^2$ and $n=1$. Then, \eqref{eq:HHOWN} can now be related to a Vessiot--Guldberg Lie algebra $V$  on $\T^*\mathbb{R} $ spanned by
 $$
 X_1= y \frac{\partial}{\partial x }
\,,\qquad X_2=\frac 12  \left(x \frac{\partial}{\partial x }-y \frac{\partial}{\partial y }\right)\,,\qquad X_{3}=- x \frac{\partial}{\partial y }
$$ spanning a Lie algebra  isomorphic to $\mathfrak{sl}_2$  of Hamiltonian vector fields relative to the symplectic form
 $
 \omega=\d x\wedge \d y\,.
 $
Hence, $X_1,X_2,X_3$ have Hamiltonian functions given by
 $$
 h_1 = \frac{y^2}{2}\,,\qquad h_2 = \frac{1}{2}\ x y \,,\qquad h_3= \frac{x ^2}{2}\,.
 $$
 relative to $\omega$ that span a Lie algebra isomorphic to $\mathfrak{sl}_2$. Now, \eqref{eq:HHOWN} can be related to  Stratonovich operator of the form
 $$
\mathfrak{H}(x ,y ,t,W)=\left(X_1+\omega^2X_3,2X_1+2\omega^2X_3\right), $$
which turns \eqref{eq:HHOWN} into a  Hamiltonian stochastic Lie system and $V$ into an associated Vessiot--Guldberg Lie algebra isomorphic to $\mathfrak{sl}_2$. Note that this system has two Hamiltonians
$$
\widetilde{h}=(y^2/2+\omega^2x^2/2,y^2 +\omega^2x^2 ).
$$
Then, $y^2/2+\omega^2/2$ is the Hamiltonian of the system. 
 The point $(0,0)$ is an equilibrium point and the function $f=y^2/2+\omega^2x^2/2$ is strongly conserved because $\{f,f\}=\{2f,f\}=0$ relative to the Poisson bracket induced by $\omega$. Since $d^2f$ is positive definite, then $(0,0)$ is almost surely stable.

\section{Relative equilibrium points and stochastic Hamiltonian Lie systems} \label{Sec:Relative_Equilibrium_Stochastic_Differential}

In short, relative equilibrium points for a system of differential equations on a manifold $M$, in general, and for a system of stochastic differential equations, in particular, are points where the dynamics is generated by Lie group actions such that the elements of the group are understood as  symmetries of the differential equations. In this section,  we define for the first time relative equilibrium points for Hamiltonian stochastic differential equations in Stratonovich form of the particular type  \eqref{eq:HamSDE} using the theory in \cite{LO08,LO09b}. 

Let us start by reviewing the notion of symmetries for stochastic differential equations and symplectic reduction for Hamiltonian stochastic differential equations (see \cite[Section 2]{LO09b} for details). 

\begin{definition}\label{def:symmetry}
Let $B:\mathbb{R}_+\times \Omega\rightarrow \mathbb{R}^\ell$ be a semi-martingale and let $\mathfrak{S}\colon\T\mathbb{R}^\ell\times M\rightarrow \T M$ be a Stratonovich operator. A diffeomorphism $\phi\colon M\rightarrow M$ is a {\it symmetry of the stochastic differential equation associated with $\mathfrak{S}$} if
$$
\mathfrak{S}(\delta B,\phi (\Gamma))=\T\phi [\mathfrak{S}(\delta B,\Gamma)]\,,\qquad \forall (\delta B,\Gamma)\in \T \mathbb{R}^\ell\times M\,.
$$
More generally, a Lie group action $\Phi:G\times M\rightarrow M$ is a {\it Lie group of symmetries of  the stochastic differential equation induced by $\mathfrak{S}$} if $\Phi_g:\Gamma\in M\mapsto \Phi(g,\Gamma)\in M$ is a symmetry of $\mathfrak{S}$ for every $g\in G$. Similarly, we say that $\phi$ (resp. $\Phi$) are {\it symmetries (resp. a Lie group of symmetries) of the Stratonovich operator $\mathfrak{S}$.}     
\end{definition}

\begin{remark} Note the slight change of notation for the Stratonovich operator in Definition \ref{def:symmetry}, where for instance the first entry is an element of $\T\mathbb{R}^\ell$ instead of an element of $\mathbb{R}^\ell$ as in previous sections.  It is worth stressing that the term `Lie group symmetries' also appears in the literature on standard and stochastic PDEs with the same meaning of `Lie group of symmetries' (cf. \cite{AGH_18}). 
\end{remark}

The relevance of the symmetries of Stratonovich operators is due to the fact that they transform a particular solution of the stochastic differential equations related to it to another particular solution of the same stochastic differential equation \cite[Theorem 2.2]{LO09b}. In the case of a Lie group of symmetries of a stochastic differential equation induced by $\mathfrak{S}$, then every element of the associated Lie group can be understood as a symmetry of the stochastic differential equation.

Let us give a $t$-dependent generalisation of the stochastic Marsden--Meyer--Weinstein reduction for Hamiltonian stochastic  systems as introduced in \cite[Theorem 3.1]{LO09b} for the case of  Hamiltonian stochastic systems \eqref{eq:HamSDE}, which essentially amounts to applying a standard reduction for each particular $t\in \mathbb{R}$.  

Let $(M,\omega)$ be a symplectic manifold and let $\Phi\colon G\times M\rightarrow M$ be a Lie group action  admitting a coadjoint equivariant momentum map $J \colon M \rightarrow \mathfrak{g}^*$ being also a Lie group of symmetries of the Hamiltonian functions $\widetilde{h}\colon\mathbb{R}\times M\rightarrow \mathbb{R}^\ell$ of a Stratonovich operator $\mathfrak{H}$. Every regular $\mu\in \mathfrak{g}^*$ of $J$ gives rise to a function $h_\mu \colon \mathbb{R}\times J^{-1}(\mu)/G_\mu \rightarrow \mathbb{R}^\ell$ on the manifold $M_\mu=J^{-1}(\mu)/G_\mu$
determined by the equality $h_\mu(t,\pi_\mu(\cdot))=h(t, \iota_\mu(\cdot))$ for every $t\in \mathbb{R}$, where $\iota_\mu\colon J^{-1}(\mu)\hookrightarrow M$ is the natural immersion,  we assume the isotropy subgroup $G_\mu$ of the coadjoint action at $\mu$ to act freely and properly on $J^{-1}(\mu)$, and  $\pi_\mu:J^{-1}(\mu)\rightarrow M_\mu$ is the natural projection. Moreover, one obtains a unique symplectic form $\omega_\mu$ on $M_\mu$ induced by the relation $\pi_\mu^*\omega_\mu=\iota_\mu^*\omega$. In turn, this induces a stochastic Hamiltonian
system on the symplectic reduced space $(M_\mu, \omega_\mu)$ whose Stratonovich operator $\mathfrak{H}_\mu \colon \T \mathbb{R}^\ell \times M_\mu \to \T M_\mu$ is
given by (cf. \cite[Theorem 3.1]{LO09b})
$$
\mathfrak{H}_\mu(\delta B, [\Gamma]))=\T_\Gamma\pi_\mu(\mathfrak{H}(\delta B,\Gamma))\,,\qquad \forall \delta B\in \T\mathbb{R}^\ell\,,\qquad \forall\Gamma\in M\,.
$$
Moreover, if $\Gamma$ is a solution  of the stochastic Hamiltonian
system associated with $\mathfrak{H}$ with initial condition $\Gamma_0\subset J^{-1}(\mu)$ at $t=0$, then so is $\Gamma_\mu := \pi_\mu(\Gamma)$
with respect to $\mathfrak{H}_\mu$, with initial condition $\pi_\mu(\Gamma_0)$ at $t=0$.

In the particular case of stochastic Hamiltonian Lie systems, the reduced stochastic equation is a stochastic Hamiltonian Lie system related to a Stratonovich operator of the form
$$
\mathfrak{H}_\mu(t, [\Gamma]))  = \left(\sum_{\alpha=1}^rb_1^\alpha(t)X_{h^\mu_\alpha}([\Gamma]),\ldots,\sum_{\alpha=1}^rb_\ell^\alpha(t)X_{h^\mu_\alpha}([\Gamma])\right),
$$
where the
$h^\mu_\alpha \in\Cinfty(M_\mu)$ are determined by the conditions  $h^\mu_\alpha \circ \pi_\mu=h_\alpha\circ \iota_\mu$ from the generators of the Lie--Hamiltonian Lie algebra of functions $h_1,\ldots,h_r$ of the initial Hamiltonian stochastic Lie system. Note that in the case of Hamiltonian stochastic Lie systems with a Vessiot--Guldberg Lie algebra $V$, we require all related Hamiltonian functions, $h_1,\ldots,h_r$ to be invariant relative to the action of the Lie group $\Phi\colon G\times M\rightarrow M$. This condition is always satisfied for the smallest Lie algebra $V^X$ as it follows from the fact that $\mathfrak{H}$ can be reduced for each time $t$. 

The above discussion suggests the following definition.

\begin{definition} Given a Hamiltonian stochastic differential equation \eqref{eq:HamSDE} on $M$, a {\it relative equilibrium point}  $\Gamma_{\rm rel}\in M$ of \eqref{eq:HamSDE} is a point such that
\begin{equation}\label{eq:RelEquH}
\mathfrak{H}_\alpha(t,\Gamma_{\rm rel})\in D_{\Gamma_{\rm rel}},\qquad \alpha=1,\ldots,\ell\,, \qquad \forall t\in \mathbb{R}\,,
\end{equation}
where $D$ is the distribution generated by the fundamental vector fields of a Lie group action of symmetries $\Phi\colon G\times M\rightarrow M$ of \eqref{eq:HamSDE} with a momentum map $J\colon M\rightarrow \mathfrak{g}^*$.
\end{definition}

To understand the above condition in terms of the better known characterisation for standard Hamiltonian systems of differential equations \cite{MS88}, consider a basis  $\{\xi_M^1,\ldots, \xi^r_M\}$ of fundamental vector fields of the Lie group action $\Phi$ corresponding to a basis  $\{\xi_1,\ldots,\xi_r$\} of $\mathfrak{g}$. Then, the relation \eqref{eq:RelEquH} amounts to saying that
$$
\mathfrak{H}(t,\Gamma_{\rm rel})=\left(\sum_{\alpha=1}^rf_1^\alpha (t)\xi^\alpha_M(\Gamma_{\rm rel}),\ldots, \sum_{\alpha=1}^rf_\ell^\alpha (t)\xi^\alpha_M(\Gamma_{\rm rel})\right)\,,\qquad \forall t\in \mathbb{R},$$
for certain $t$-dependent functions $f_i^\alpha\in \Cinfty(\mathbb{R})$ for $i=1,\ldots,\ell$ and $\alpha=1,\ldots,r$.

Let $\mu=J(\Gamma_{\rm rel})$ in the above definition. If the Hamiltonian stochastic differential equation \eqref{eq:HamSDE} can be reduced, then $\Gamma_{\rm rel}\in M$ is contained in $J^{-1}(\mu)$. If $J^{-1}(\mu)$ is a submanifold of $M$, the evolution of $\mathfrak{H}$ is contained within $J^{-1}(\mu)$. This implies that the right-hand side of  \eqref{eq:RelEquH} belongs indeed to $D_{\Gamma_{\rm rel}}\cap \T_{\Gamma_{\rm rel}}J^{-1}(\mu)$. Moreover, the restriction of $\mathfrak{H}$ to it can be reduced to  $J^{-1}(\mu)/G_\mu$ and the $\mathfrak{H}_\mu(t,\pi_\mu(J^{-1}(\mu)))=0$. In other words, one has the following proposition, whose inverse is analogue to its deterministic case \cite{LZ21}.

\begin{proposition} Let $\Gamma_{\rm rel}\in M$ be a relative equilibrium point for a stochastic Hamiltonian system \eqref{eq:HamSDE} and let $J(\Gamma_{\rm rel})=\mu$ be a regular point of $J$. Let $G_\mu$ act on $J^{-1}(\mu)$ freely and properly so that $J^{-1}(\mu)/G_\mu$ becomes a manifold. Then, the reduction of \eqref{eq:HamSDE} to $M_\mu=J^{-1}(\mu)/G_\mu$  is such that $\pi_\mu(\Gamma_{\rm rel})$ is an equilibrium point of the reduced system and vice versa.
\end{proposition}

Finally, let us characterise relative equilibrium points for stochastic Hamiltonian systems.

\begin{theorem}\label{RelativeEquilibrum}
{\bf (Stochastic Relative Equilibrium Theorem)} A point $\Gamma_{\rm rel}\in M$ is a relative equilibrium point for a Hamiltonian stochastic differential equation \eqref{eq:HamStoLieSys} related to $\widetilde{h}:\mathbb{R}\times M\rightarrow \mathbb{R}^\ell$ if and only if there exist elements $\xi_t=(\xi^1_t,\ldots,\xi_t^\ell)\in\mathfrak{g}^\ell$ for $t\in \mathbb{R}$ such that $(\xi_{t},\Gamma_{\rm rel})$ are critical points of the coordinates of the functions $\widehat{h}_t:\mathfrak{g}^\ell\times M\rightarrow \mathbb{R}^\ell$, given by
\[\widehat{h}_t^\alpha(\xi_t,\Gamma)=\widetilde{h}^\alpha(t,\Gamma)-\left\langle {J}(\Gamma) - \mu_{e},\ \xi_t^\alpha\right\rangle,\qquad \alpha=1,\ldots,\ell,
\] 
for $\mu_e:=J(\Gamma_e)$.
\end{theorem}
\begin{proof}Let us prove the direct part. At the points $(\xi_t,\Gamma_{\rm rel})$,  the coordinates of the functions $\widehat{h}_t(\xi_t,\cdot):M\rightarrow \mathbb{R}^\ell$ have a critical point $\Gamma_{\rm rel}\in M$. Hence, each coordinate of $\widehat{h}_t(\xi_t,\cdot)$ satisfies that $\d(\widetilde{h}_t^\alpha-\langle J(\cdot),\xi^\alpha_t\rangle)|_{\Gamma_{\rm rel}}=0$. This implies that $(X_{\widetilde{h}_t^\alpha})({\Gamma_{\rm rel})=X_{J_{\xi^\alpha_t}}}(\Gamma_{\rm rel})$ and the point $\Gamma_{\rm rel}$ is a relative equilibrium point of $\mathfrak{H}$.  Conversely, if $\Gamma_{\rm rel}$ is a relative equilibrium point to $\mathfrak{H}$, at each coordinate of $\mathfrak{H}$, one has that $(X_{\widetilde{h}^\alpha})_t(\Gamma_{\rm rel})=(\xi_t)_M^\alpha(\Gamma_{\rm rel})$ for certain $\xi_t^1,\ldots,\xi_t^\ell\in \mathfrak{g}$ with $t\in \mathbb{R}$. Consequently, one has that $X_{\widetilde{h}_t^\alpha-\langle J,\xi^\alpha_t\rangle}=0$. This implies that each one of these vector fields has a Hamiltonian function given by $\widetilde{h}_t^\alpha-\langle J-\mu_e,\xi_t^\alpha\rangle$ which has a critical point at the given point $\Gamma_{\rm rel}$. 
\end{proof}


Let us consider a particular case of the the stochastic differential equation on  $\mathbb{R}^n$ with a Wiener process $W_1$ given in \eqref{eq:HHOWN} for $2a(t)=a'(t)=2$, $2b(t)=b'(t)=-2\omega^2$, any $n$,  and a constant $\omega$ (see \cite{Co_12} for similar models). Then, \eqref{eq:HHOWN} can now be related to a Vessiot--Guldberg Lie algebra $V$  on $\T^*\mathbb{R}^n $ spanned by the vector fields \eqref{eq:StoHarOsc} isomorphic to $\mathfrak{sl}_2$ and consisting of Hamiltonian vector fields relative to the symplectic form $
 \omega=\sum_{i=1}^n\d x_i\wedge \d y_i\,.
 $ Hence, $X_1,X_2,X_3$ have Hamiltonian functions given by \eqref{Eq:StoOscHamFun} relative to $\omega$ that span a Lie algebra of functions isomorphic to $\mathfrak{sl}_2$. Now, \eqref{eq:HHOWN} can be related to  Stratonovich operator of the form
 $$
\mathfrak{H}(x ,y ,t,W)=\left(X_1+\omega^2X_3,2X_1+2\omega^2X_3\right), $$
which turns \eqref{eq:HHOWN} into a  Hamiltonian stochastic Lie system and $V$ into an associated Vessiot--Guldberg Lie algebra. Note that this system has two Hamiltonians
$$
\widetilde{h}=(h_1,h_2)=\left(\sum_{i=1}^n (y_i^2/2+\omega^2x_i^2/2),\sum_{i=1}^n (y_i^2 +\omega^2x_i^2 )\right).
$$
Let us consider the Lie symmetries $Y_k=-\omega^2x_k\frac{\partial}{\partial y_k}+y_k\frac{\partial}{\partial x_k}$ for $k=1,\ldots,s$ of the above Hamiltonian system. In other words, $Y_1,\ldots,Y_s$ are Hamiltonian relative to $\omega$ and are Lie symmetries of $\mathfrak{H}$. Moreover, $Y_1,\ldots,Y_s$ commute between themselves giving rise to an Abelian Lie algebra of vector fields whose integration gives rise to a Lie group action on $\T\mathbb{R}^n$ and an associated momentum map 
\[{\rm J}:(x_1,y_1,\ldots,x_n,y_n)\in \T\mathbb{R}^n \mapsto (x_1^2/2+\omega^2 y_1^2/2,\ldots,x_s^2/2+\omega^2y_s^2/2)\in \mathbb{R}^{s*}.\] 
Since  $Y_1,\ldots,Y_s$ are Hamiltonian Lie symmetries of $(h_1,h_2)$, the system can be reduced. Following the approach in \cite{LO09b}, one projects the Hamiltonian stochastic system on  $J^{-1}(c_1,\ldots,c_s)$ with $\prod_{i=1}^sc_i\neq 0$. The latter condition is employed to ensure that $(c_1,\ldots,c_s)$ is a regular value of $J$ and $J^{-1}(c_1,\ldots,c_s)$ is a submanifold. Note that $Y_1\wedge \ldots\wedge Y_s\neq 0$ and $Y_1,\ldots,Y_s$ are tangent to $J^{-1}(c_1,\ldots,c_s)$, which gives rise to a reduced stochastic system on $J^{-1}(c_1,\ldots,c_s)/\mathbb{R}^s\simeq \T\mathbb{R}^{n-s}$, which is Hamiltonian relative to a symplectic form $\omega=\sum_{i=s+1}^n\d x_i\wedge \d y_i$. More exactly, the reduced stochastic system reads \begin{equation}\label{eq:HHOWNR}
\delta\begin{pmatrix}
    x_i \\ y_i
\end{pmatrix} = \begin{pmatrix}
    0& 1 \\ -\omega^2 & 0\end{pmatrix} \begin{pmatrix}
    x_i \\ y_i
\end{pmatrix} \delta t + \begin{pmatrix}
0 & 2 \\ -2\omega^2 & 0
\end{pmatrix} \begin{pmatrix}
    x_i \\ y_i
\end{pmatrix} \circ \delta W_1\,,\qquad i=s+1,\ldots,n\,.
\end{equation}
Note that the equilibrium points of the above system are those points such that  $x_{s+1}=y_{s+1}=\ldots=x_n=y_n=0$, which are the projections of the original ones of the form
$$
(x_1,y_1,\ldots,x_s,y_s,\stackrel{(n-s)-{\rm  pairs}}{\overbrace{0,0,\ldots,0,0}}).
$$
Note that these points are indeed the relative equilibrium points of our initial stochastic system, where the components of $\mathfrak{H}$ can be written as a linear combination of $Y_1,\ldots,Y_s$. Indeed, at these points $X_1+\omega^2 X_3=\sum_{i=1}^sY_i$ and $2X_1+2\omega^2X_3=\sum_{i=1}^s2Y_i$. As stated in Proposition \ref{RelativeEquilibrum}, one has that at these points $h_1-\sum_{i=1}^sx_i^2/2+\omega^2y_i^2/2$ and $h_2-\sum_{i=1}^sx_i^2+\omega^2y_i^2$ have equilibrium points.  

\section{Stochastic Poisson coalgebra method}\label{Se:PCSLS}

Let us review and extend the Poisson coalgebra method for Hamiltonian systems relative to a symplectic form. In general, our theory is a stochastic generalisation of what can be found in the classical setting \cite[Section 4.2.7]{LS21}. Although the procedure is very similar to the original approach, a few key differences allow its application to many new domains. In particular, the method can  still be applied to Hamiltonian stochastic Lie systems by considering that they are determined by a certain $\ell$-family of $t$-dependent Hamiltonian functions. Indeed, our procedure is a new modification of the coalgebra method for deriving superposition rules for $k$-symplectic Lie systems (see \cite[Section 7.8]{LS21} and references therein). 

It is convenient to stress that the proof of the stochastic Lie theorem shows that a superposition rule for stochastic Lie systems, in Stratonovich form, can be obtained in a similar manner to the case of deterministic Lie systems, namely (see \cite{CGM07,Dissertationes,LS21} for
	details and examples):
	\begin{enumerate}
		\item Consider a Vessiot--Guldberg Lie algebra of the stochastic Lie system spanned by a basis of vector fields $X_1,\ldots,X_r$ on the manifold $M$.
  \item Find the smallest natural number $m\in \mathbb{N}$, so that $X^{[m]}_1,\ldots,X_r^{[m]}$  
		are linearly independent at a generic point.
		\item Use local coordinates $x^1,\ldots,x^n$ on $M$ and consider this coordinate
		system to be defined on each copy of $M$ within $M^{m+1}$ to get a coordinate system
		$\{x^i_{(a)}\mid i=1,\ldots,n,\,\,a=0,\ldots,m\}$ on $M^{m+1}$.
		Obtain first integrals $F_1,\ldots, F_n$ common to all the diagonal prolongations
		$X^{[m+1]}_1,\ldots,  X^{[m+1]}_r$ such that 
		\begin{equation}\label{Cond}
		\frac{\partial(F_1,\ldots,F_n)}{\partial(x_{(0)}^1,\ldots,x_{(0)}^n)}\neq 0.
		\end{equation}
		  
		\item 
	    Condition \eqref{Cond} allows us to ensure that the equations $F_i=k_i$, for $i=1,\ldots,n$, enable us to write the expressions of the variables $x_{(0)}^1\ldots,x_{(0)}^n$ in
		terms of $x^1_{(a)},\ldots,x_{(a)}^n$, with $a=1,\ldots,m,$  and $k_1,\ldots,k_n$.
		\item The obtained expressions lead to a superposition rule depending on a generic family of $m$ particular solutions
		and the constants $k_1,\ldots, k_n$.  This holds true even in the Stratonovich stochastic realm. 
	\end{enumerate}

Note that every Hamiltonian stochastic Lie system is related to a Stratonovich operator $\mathfrak{H}$ given by $\ell$ components and each one can be understood as a $t$-dependent vector field. Hence, $\mathfrak{H}$ can be understood as a $t$-dependent $\ell$-vector field, i.e. for every $t\in \mathbb{R}$, it is a section of the $\ell$-tangent bundle of velocities $\T^\ell M=\T M\oplus \dotsb \oplus \T M\rightarrow M$, where $\T M\oplus\dotsb\oplus \T M$ is to be considered as the Whitney sum of the tangent bundle to $M$ with itself. Similarly, one can prolong diagonally $\mathfrak{H}$ to a map $\mathfrak{H}^{[m]}:\T \mathbb{R}^
\ell\times M^m\rightarrow \T M^m$. 
Moreover, $\mathfrak{H}$ is related to a $t$-dependent Hamiltonian function $h\colon\mathbb{R}\times M\rightarrow \mathbb{R}^\ell$ which gives rise, by  prolonging diagonally each of its components for every fixed $t$,  to a mapping $h^{[m]}:\mathbb{R}\times M^m\rightarrow \mathbb{R}^\ell$. Additionally, the diagonal prolongation $h^{[m]}$ is the  $t$-dependent Hamiltonian function related to $\mathfrak{H}^{[m]}$. Then, the following result is immediate.

\begin{proposition}\label{Cor:LieAlgJL2}
    If $\mathfrak{H}$ is a Hamiltonian stochastic Lie system admitting a $t$-dependent Hamiltonian $h\colon\mathbb{R}\times M\rightarrow \mathbb{R}^\ell$ relative to a symplectic form $\omega$, then $\mathfrak{H}^{[m]}$ is a Hamiltonian stochastic Lie system relative to the symplectic form $\omega^{[m]}$ admitting a $t$-dependent Hamiltonian $h^{[m]}\colon\mathbb{R}\times M^m\rightarrow \mathbb{R}^\ell$. In particular, if $h_1,\ldots,h_r$ is a basis of a Lie--Hamilton Lie algebra  for $\mathfrak{H}$, then $h_1^{[m]},\ldots,h_r^{[m]}$ is a basis of a Lie--Hamilton Lie algebra for $\mathfrak{H}^{[m]}$.
\end{proposition}

We have the following immediate proposition.
\begin{proposition} The space of $\ell$-Hamiltonian functions, namely $\Cinfty(M,
\mathbb{R}^
\ell)$ on a symplectic manifold $(M,\omega)$, with Poisson bracket $\{\cdot,\cdot\}_\omega$ induced by $\omega$, is a Poisson algebra relative to the bracket
$$
\{h,h'\}_\ell=(\{h_1,h'_1\}_\omega,\ldots,\{h_\ell,h'_\ell\}_\omega)\,, \quad \forall h=(h_1,\ldots,h_\ell),\quad h'=(h'_1,\ldots,h'_\ell)\in \Cinfty(M,\mathbb{R}^\ell)\,,
$$
and the multiplication
$$
h\cdot h'=(h_1h_1',\ldots,h_\ell h'_\ell)\,,\qquad 
\forall h,h'\in \Cinfty(M,\mathbb{R}^\ell)\,.
$$
\end{proposition}
It is worth recalling that the space of $\ell$-Hamiltonian functions relative to an $\ell$-symplectic form was not a Poisson algebra. This was due to the fact that the multiplication of $\ell$-Hamiltonian functions could not be ensured to be an $\ell$-Hamiltonian function as each coordinate is assumed to be a Hamiltonian function of the same vector field and this cannot be ensured for their bracket as each coordinate of the bracket is associated with different presymplectic forms and may be related to different vector fields  (see \cite[Chapter 7]{LS21}). Here, this problem does not appear, since each $h_i$, with $i=1,\ldots,\ell$, may be the Hamiltonian function of a different vector field.

For the sake of completeness, let us prove the following result. Recall that ${\rm Lie}(\{h_t\}_{t\in \R},\{\cdot,\cdot\}_\omega)$ is the smallest Lie algebra of Hamiltonian functions relative to $\omega$ containing $\{h_t\}_{t\in \mathbb{R}}$.

\begin{proposition}
    Let $\mathfrak{H}$ be a Hamiltonian stochastic Lie system with respect to a symplectic form $\omega$ and possessing a Lie--Hamilton  Lie algebra $\mathfrak{W}=\langle h_1,\ldots,h_r\rangle$ relative to the Poisson bracket related to $\omega$. A function $f\in \Cinfty(M)$ is a strong constant of motion for $\mathfrak{H}$ if  commutes with all the elements of ${\rm Lie}(\{h_t\}_{t\in \R},\{\cdot,\cdot\}_\omega)$. In particular, $f$ is a strong contact of the motion if it commutes with the elements of $\mathfrak{M}$. 
\end{proposition}
\begin{proof} The function $f$ is a strong constant of motion for $\mathfrak{H}$ if
\begin{equation}\label{Eq:Cons}
0=X^i_tf\,,\qquad \forall t\in \mathbb{R}\,,\qquad i=1,\ldots,\ell,
\end{equation}
where the $X_t^i$ are the components at a fixed $t$ of the $t$-dependent $\ell$-vector field related to $\mathfrak{H}$.  Since $X_t^if=\{h^i_t,f\}_\omega$ for $i=1,\ldots,\ell$ and $t\in \mathbb{R}$, the result follows. Note also that
$$
\{f,\{h_t,h_{t'}\}_\omega\}_\omega=\{\{f,h_t\},h_{t'}\}_\omega\}_\omega+\{h_t,\{f,h_{t'}\}_\omega\}_\omega\,,\qquad \forall t,t'\in \mathbb{R}\,.
$$
Inductively, $f$ is commutes with all the elements of ${\rm Lie}(\{h_t\}_{t\in \mathbb{R}},\{\cdot,\cdot\}_\omega)$. Since we restrict ourselves to the case for a symplectic form $\omega$, one has that ${\rm Lie}(\{h_t\}_{t\in \mathbb{R}},\{\cdot,\cdot\}_\omega)$ is included in $ \mathfrak{W}\oplus \mathbb{R}$, where $\mathbb{R}$ stands for the space of constant functions. And the latter ensures that $f$ is a strong constant of motion.
 
\end{proof}

In Lie--Hamilton systems, a $t$-dependent Hamiltonian vector field admits for every $t\in \mathbb{R}$ a Hamiltonian function belonging to a finite-dimensional Lie algebra of Hamiltonian functions relative to a Poisson bracket.

\begin{proposition}\label{Prop:6.7}
    Let $\mathfrak{H}$ be a Hamiltonian stochastic Lie system with an associated Lie--Hamilton  Lie algebra $(\mathfrak{W}=\langle h_1,\ldots,h_r\rangle, \{\cdot,\cdot\}_\omega)$ relative to the symplectic form $\omega$. Let $\{v_1,\ldots ,v_r\}$ be a basis of linear coordinates on $\mathfrak{g}^*\simeq \mathfrak{W}$. Given the  momentum map $J : \Gamma\in M\mapsto h_i(\Gamma)=(J^*v^i)(\Gamma)\in\mathfrak{g}^\ast$, the pull-back $J^\ast C$ of any Casimir function $C$ on $\mathfrak{g}^\ast$ is a constant of motion for $\mathfrak{H}$. Moreover, if  $C=C(v_1,\ldots,v_r)$, then 
    \begin{equation}\label{Eq:ExCas}
C\left(\sum_{a=1}^kh_1(x_{(a)}),\ldots, \sum_{a=1}^kh_r(x_{(a)})\right)\,,\qquad 1\leq k\leq m, 
\end{equation} is a constant of motion of $\mathfrak{H}^{[m]}$.
\end{proposition}

The stochastic Poisson coalgebra method above takes its name from the fact that it is applied to Hamiltonian stochastic Lie systems and analyses the use of Poisson coalgebras and a so-called {\it coproduct} to obtain superposition rules. In fact, the coproduct is responsible for the form of \eqref{Eq:ExCas} (see \cite{LS21}). Although we have focused on Hamiltonian symplectic systems, the above Poisson coalgebra method can also be applied to Hamiltonian systems relative to many other geometric structures.

Let us provide a simple example illustrating the above techniques based on the system
\begin{equation}\label{eq:HHOWNRANew}
\delta\begin{pmatrix}
    x_i \\ y_i
\end{pmatrix} = \left[\begin{pmatrix}
    0& 1 \\ -\omega^2(t) & 0\end{pmatrix} \begin{pmatrix}
    x_i \\ y_i
\end{pmatrix}+\begin{pmatrix}
0 \\
g(t)
\end{pmatrix}\right] \delta t + \begin{pmatrix}
0 \\
\alpha(t)
\end{pmatrix}  \circ \delta W_1,
\end{equation}
for a Brownian motion $W_1$, and any $t$-dependent functions $g(t)$, $\omega(t)$, and $\alpha(t)$, which retrieves certain oscillators in \cite{Co_12} and is connected to numerous physical oscillators with driven stochastic components. Since 
$$
\mathfrak{H}=(\omega^2(t)X_3+X_1+g(t)X_5,\alpha(t)X_5),
$$
system \eqref{eq:HHOWNRANew} admits a Vessiot--Guldberg Lie algebra of Hamiltonian vector fields $V_o$ relative to $\omega=\sum_{i=1}^n\d x_i\wedge \d y_i$ spanned by
$$
 X_1= \sum_{i=1}^ny_i \frac{\partial}{\partial x_i }
\,,\qquad X_2=\frac 12 \sum_{i=1}^n  \left(x_i \frac{\partial}{\partial x_i }-y_i \frac{\partial}{\partial y_i }\right)\,,\qquad X_{3}=-\sum_{i=1}^n x_i \frac{\partial}{\partial y_i },\qquad 
$$
$$X_4=\sum_{i=1}^n\frac{\partial}{\partial x_i},\qquad X_5=\sum_{i=1}^n\frac{\partial}{\partial y_i}.
$$
Their non-vanishing commutation relations read
\[
[X_{1},X_{2}]=X_{1},\qquad [X_{1},X_{3}]=2X_{2},\qquad [X_{2},X_{3}]=X_{3},
\]
\[
[X_{1},X_{5}]=-X_{4},\qquad [X_{3},X_4]=X_5,\qquad [X_2,X_4]=-\frac 12X_4,\qquad [X_2,X_5]=\frac 12 X_5.
\]
Therefore, the Lie algebra  $V_o$  is isomorphic to the semi-direct sum of Lie algebras  $\mathfrak{sl}_2\ltimes \mathbb{R}^2$. In particular, $\langle X_1,X_2,X_3\rangle$ is isomorphic to $\mathfrak{sl}_2$, while $\langle X_4,X_5\rangle$ is an Abelian ideal of $V_o$ . 
The vector fields $X_1,\ldots,X_5$ are related to a Lie algebra of Hamiltonian functions spanned by
\begin{equation}\label{Eq:two-photon}
h_1=\frac 12\sum_{i=1}^ny_i^2,\quad h_2=\frac 12\sum_{i=1}^n x_iy_i,\quad h_3=\frac 12\sum_{i=1}^nx_i^2,\quad h_4=\sum_{i=1}^ny_i,\quad h_5=-\sum_{i=1}^nx_i,\quad h_6=n,
\end{equation}
Hence, they span a Lie algebra $\mathfrak{W}_0$  isomorphic to the so-called two-photon algebra \cite{BH_01}. Indeed, $\langle h_1,h_2,h_3\rangle$ is isomorphic to $\mathfrak{sl}_2$, while $\langle h_4,h_5,h_6\rangle$ is a three-dimensional Heisenberg algebra  $\mathfrak{h}_3$ and the LH algebra $\mathfrak{W}_0$ is isomorphic to $\mathfrak{g}_0\simeq \mathfrak{sl}_2\ltimes\mathfrak{h}_3$. This gives rise to a momentum map $J_n:\Gamma\in \T\mathbb{R}^n\mapsto h_1(\Gamma)e^1+\ldots +h_6(\Gamma)e^6\in \mathfrak{g}_0^*$. Note that $\{e_1,\ldots,e_6\}$ is a basis of $\mathfrak{g}_0$ spanning the same commutation relations than \eqref{Eq:two-photon}. Then, one has the Casimir of $\mathfrak{g}_0$ given by \footnote{The isomorphism $(h_0, h_1, h_2, h_3, h_4, h_5)\leftrightarrow (e_6, -e_5, e_4, -2e_2, e_3, -e_1)$ from the Hamiltonian functions in \cite[p. 27]{BHLS_15} maps the Casimir in \cite{BHLS_15} into $C$.
}
$$
C=2(-e_5^2e_1-e_4^2e_3-2e_5e_4e_2)-4e_6(e_2^2-e_3e_1),
$$
which gives rise to a constant of motion of \eqref{eq:HHOWNRANew} of the form
$$
F_n=J_n^*C=2(-h_5^2h_1-h_4^2h_3-2h_5h_4h_2)-4h_6(h_2^2-h_3h_1).
$$
One can also see that $h_1,\ldots,h_6$ are the diagonal prolongations to $(\T\mathbb{R})^n$ of the oscillator \eqref{eq:HHOWNRANew} for $n=1$. Then, $F_n$ is nothing but the constant of motion \eqref{Eq:ExCas} given for this specific case. Moreover,
$$
F_1=F_2=0,\qquad F_3=(x_3(y_2-y)+x_2(y_1-y_3)+x_1(y_3-y_2))^2
$$ while further functions $F_n$ for $n>3$ have a complicated expression.

\section{Conclusions and outlook}\label{Eq:Conclusions}

Our paper provided an introduction to stochastic Lie systems that could draw the attention of people working on stochastic systems or Lie systems. In this respect, some basic notions and results in stochastic systems and Lie systems have been explained in detail. In particular, some types of semi-martingales, which are suitable for use  in the theory, have been applied. Meanwhile, some differential geometry concepts and other ideas for Lie systems have been explained through initial examples. This paper has skipped many technical details of the introduction to stochastic Lie systems in \cite{LO09} that can be omitted in practical applications. This has made our presentation more accessible to a general public. We also reviewed the theory of stochastic Lie systems, correcting a mistake in the literature to give a solid foundation for the theory. Some relations between the Stratonovich and the It\^o approaches to stochastic Lie systems have been reviewed. This showed that the stochastic Lie theorem in the It\^o framework is quite different from its classical form.
Moreover, we proved that this fact is relevant when considering practical applications of the theory.

Types and generalisations of stochastic Lie systems have been introduced. In particular, we have focused on the study of Hamiltonian stochastic Lie systems. The theory of stability and relative equilibrium points for Hamiltonian stochastic differential equations has been developed and, in particular, we have focused in the study of Hamiltonian stochastic Lie systems relative to symplectic forms. Several examples with potential applications have been provided so as to illustrate the theory.

Some results concerning the energy-momentum method for stochastic Lie systems have been obtained, generalising \cite{Bz_14}. Many examples concerning physical/epidemiological applications have been developed. We have also remarked that the theory of Poisson coalgebras can be developed for our Hamiltonian stochastic Lie systems.

Concerning applications, stochastic SIS models have been analysed. Such models are usually studied  using deterministic methods in the literature. Instead, we approach them here in a pure stochastic way. Many other physical systems, like different  types of stochastic oscillators with damped terms, have also been analysed.

In the future, we aim to further develop the energy-momentum method for stochastic Lie systems. This entails the search for criteria ensuring the stability of equilibrium points, for a Hamiltonian Stratonovich operator,  after reduction. Moreover, we expect, as in the deterministic case, to obtain certain degeneracies of the Hamiltonian functions that will need to be analysed to solve the problem \cite{MS88}.  Note that a further study of the advantages of epidemiological models from a purely stochastic point of view will need to be undertaken. In particular, we plan to study SIS models that have been traditionally studied in a deterministic Hamiltonian manner \cite{EFSZ20,EFSZ22} via stochastic Lie systems, and analyse their possible advantages. There seems to be a quite large new field of applications for Hamiltonian stochastic Lie systems or Hamiltonian foliated stochastic Lie systems \cite{AGH_18,LZ21}. Moreover, we also plan to study a stochastic Hamiltonian system appearing in celestial mechanics concerning the stochastic variation of the inertia tensor \cite{Sa70}, biological methods, epidemiological models, coronavirus systems, etc. Additionally, it would be interesting to study other types of superposition rules appearing in stochastic differential equations \cite{LSZ_23} and to apply the Poisson coalgebra method to particular Hamiltonian stochastic Lie systems. Finally, we would like to analyse the potential extensions of methods designed for Lie systems, or their generalisations, to stochastic counterparts. For instance, PDE Lie systems appear in hydrodynamic equations in \cite{GL_25} and the analysis of conditional symmetries in \cite{GL_19}. It would be interesting to extend, first, the notion of a stochastic Lie system to PDE Lie systems \cite{Dissertationes,Ram_02}. Moreover, one of the authors of this work along with his colleagues are studying hydrodynamical equations with a stochastic character, e.g. stochastic Burgers equations \cite{GL_25,JWL_24}, so as to generalise the Riemann invariant method \cite{GL_23} to a stochastic realm, and conditional and standard Lie symmetries for stochastic partial differential equations, partially attempting to extend the methods of \cite{GL_19} and, potentially, other works \cite{OBPS_20}.

\section*{Acknowledgements}
\addcontentsline{toc}{section}{Acknowledgements}
We acknowledge fruitful discussions with J. P. Ortega on establishing a correction for his stochastic Lie theorem. We express our gratitude to two anonymous referees for their valuable feedback, which motivated us to enhance our manuscript considerably by refining and expanding examples, adding new insights into our theoretical framework, and offering suggestions to make it more understandable to a broader readership. 
X. Rivas acknowledges partial financial support from the Spanish Ministry of Science and Innovation grants PID2021-125515NB-C21 and RED2022-134301-T of AEI, and the Ministry of Research and Universities of the Catalan Government project 2021 SGR 00603. The research of Marcin Zaj\k ac was funded by the Polish National Science Centre under a PRELUDIUM project with contract number UMO-2021/41/N/ST1/02908. For the purpose of Open Access, the authors have applied a CC-BY public copyright licence to any
Author Accepted Manuscript (AAM) version arising from this submission. The authors also acknowledge the contribution of RED2022-134301-T funded by MCIN/AEI/10.13039/501100011033 (Spain). 


\bibliographystyle{abbrv}
\addcontentsline{toc}{section}{References}
\itemsep 0pt plus 1pt
\bibliography{references.bib}

\end{document}